\documentclass[a4paper]{amsproc}
\usepackage{amssymb,amsmath,amsfonts,amsthm,epsfig,amscd,hyperref,mathtools}
\usepackage{stmaryrd,comment,enumerate}
\usepackage[all,cmtip,poly]{xy}
\usepackage{svn, url}
\usepackage{mathrsfs,tikz-cd}
\usepackage{xr}

\newtheorem{theorem}[subsubsection]{Theorem}
\newtheorem{expected theorem}[subsection]{Expected Theorem}
\newtheorem{thm}[subsubsection]{Theorem}
\newtheorem{lemma}[subsubsection]{Lemma}
\newtheorem{lem}[subsubsection]{Lemma}

\newtheorem{thm*}{Theorem}%

\makeatletter
\expandafter\newcommand\csname thethm*default\endcsname{\thethm*}
\newcommand{\thmstarnum}[1]{\expandafter\gdef\csname thethm*\endcsname{#1*}}
\thmstarnum{\thethm}
\expandafter\g@addto@macro\csname endthm*\endcsname{\thmstarnum{\thethm}}
\makeatother

\newtheorem{cor}[subsubsection]{Corollary}
\newtheorem{conj}[subsubsection]{Conjecture}

\newtheorem{prop}[subsubsection]{Proposition}

\newtheorem{defn}[subsubsection]{Definition}

\theoremstyle{remark}
\newtheorem{remark}[subsubsection]{Remark}
\newtheorem{rem}[subsubsection]{Remark}

\newtheorem{exercise}[subsubsection]{Exercise}

\numberwithin{equation}{subsection}

			\newif\iffinalrun
			\iffinalrun
			\else
			\fi

			\iffinalrun
			\newcommand{\need}[1]{}
			\newcommand{\mar}[1]{}
			\else
			\newcommand{\need}[1]{{\tiny *** #1}}
			\newcommand{\mar}[1]{\marginpar{\raggedright\tiny #1}}
			\fi

			\newcommand{\A}{\AA}
			\newcommand{\C}{\CC}
			\newcommand{\F}{\FF}
			\newcommand{\N}{\mathcal{N}}
			\newcommand{\Q}{\QQ}
			\newcommand{\R}{\RR}
			\newcommand{\Z}{\ZZ}

			\renewcommand{\O}{\cO}
			
			\newcommand{\m}{\frakm}

			\newcommand{\sm}{\mathrm{sm}}

			\renewcommand{\AA}{{\mathbb A}}
			
			\newcommand{\CC}{{\mathbb C}}

			\newcommand{\FF}{{\mathbb F}}
			\newcommand{\GG}{{\mathbb G}}

			\newcommand{\QQ}{{\mathbb Q}}
			\newcommand{\RR}{{\mathbb R}}
			
			\newcommand{\TT}{{\mathbb T}}

			\newcommand{\ZZ}{{\mathbb Z}}

			\renewcommand{\bf}{\ensuremath{\mathbf{f}}}

			\newcommand{\cA}{{\mathcal A}}
			
			\newcommand{\cC}{{\mathcal C}}
			
			\newcommand{\cE}{{\mathcal E}}

			\newcommand{\cH}{{\mathcal H}}

			\newcommand{\cL}{{\mathcal L}}
			
			\newcommand{\cN}{{\mathcal N}}
			\newcommand{\cO}{{\mathcal O}}

			\newcommand{\cU}{{\mathcal U}}
			\newcommand{\cV}{{\mathcal V}}
			\newcommand{\cW}{{\mathcal W}}
			\newcommand{\cX}{{\mathcal X}}
			\newcommand{\cY}{{\mathcal Y}}
			\newcommand{\cZ}{{\mathcal Z}}
			
			\newcommand{\frakm}{\mathfrak{m}}

			\newcommand{\frakX}{\mathfrak{X}}

			\newcommand{\Fbar}{\overline{\F}}
			\newcommand{\Qbar}{\overline{\Q}}

			\newcommand{\Fp}{\F_p}

			\newcommand{\Fpbar}{\Fbar_p}

			\newcommand{\Zl}{\Z_{\ell}}
			\newcommand{\Zp}{\Z_p}

			\newcommand{\Zpx}{\Zp^{\times}}

			\newcommand{\Ql}{\Q_{\ell}}
			\newcommand{\Qp}{\Q_p}

			\newcommand{\Qpbar}{\Qbar_p}

			\newcommand{\Qpx}{\Qp^{\times}}

			\DeclareMathOperator{\End}{End}

			\DeclareMathOperator{\Gal}{Gal}
			\DeclareMathOperator{\GL}{GL}

			\DeclareMathOperator{\Hom}{Hom}
			\DeclareMathOperator{\im}{im}
			\DeclareMathOperator{\Ind}{Ind}
			
			\DeclareMathOperator{\Max}{Max}

			\DeclareMathOperator{\SL}{SL}
			\DeclareMathOperator{\SO}{SO}
			
			\DeclareMathOperator{\Spec}{Spec}
			
			\DeclareMathOperator{\Spf}{Spf}
			\DeclareMathOperator{\Spa}{Spa}
			
			\DeclareMathOperator{\Sym}{Sym}

			\newcommand{\Frob}{\mathrm{Frob}}

			\newcommand{\et}{\mathrm{\acute{e}t}}

			\newcommand{\lb}{[\![}
			\newcommand{\rb}{]\!]}
			\newcommand{\JL}{\operatorname{JL}}

			\newcommand{\Gm}{\GG_m}

			\newcommand{\ord}{\mathrm{ord}}
			
			\newcommand{\ad}{\mathrm{ad}}

			\newcommand{\Iw}{\mathrm{Iw}}
			\newcommand{\id}{\mathrm{id}}
			\usepackage{xcolor}
			\hypersetup{
				colorlinks
			}

\begin{document}

\title{Construction of eigenvarieties}
\author{James Newton}
\email{newton@maths.ox.ac.uk}
\address{Mathematical Institute, University of Oxford, Woodstock Road, Oxford OX2 6GG, UK}

\begin{abstract}
These are notes based on four lectures given at the Heidelberg spring school on non-archimedean geometry and eigenvarieties. None of the
contents are original work. Our goal is to explain the construction of eigenvarieties in various different contexts, including the prototypical example of the Coleman--Mazur eigencurve. We will also discuss some of the common geometric properties of eigenvarieties.
\end{abstract}
\maketitle
\tableofcontents

\section*{Acknowledgements}
We thank the participants of the Heidelberg spring school for their interest and questions during the lectures, along with the organizers of the school. Thanks to Eugen Hellmann, Christian Johansson and Judith Ludwig for helpful discussions before, during and after the school. Thanks to an anonymous referee, H{\aa}vard Damm-Johnsen and Zachary Feng for their comments on an earlier version of these notes. Thanks also to H{\aa}vard for making the image in section \ref{sec:mc}.

\section*{Funding}
For the purpose of Open Access, the author has applied a CC BY public copyright licence to any Author Accepted Manuscript (AAM) version arising from this submission. The author was supported by a UKRI Future Leaders Fellowship, grant MR/V021931/1. 

Final editing of these notes was completed whilst participating in the Hausdorff Institute trimester, `The Arithmetic of the Langlands Program', funded by the Deutsche Forschungsgemeinschaft (DFG, German Research Foundation) under Germany's Excellence Strategy -- EXC-2047/1 -- 390685813.

\section{Introduction}
The goal of these lectures is to explain some examples of constructions of eigenvarieties, together with some of the properties of these spaces. We will apply results from Ludwig's lectures \cite{jl-heidelberg}, which develop crucial ingredients from $p$-adic analysis and geometry.

The main focus of the notes will be on two key examples. In Section \ref{sec:ocmfs} we will discuss the most classical setting of overconvergent modular forms and the Coleman--Mazur eigencurve. In Section \ref{sec:defquat} we will discuss overconvergent automorphic forms for definite quaternion algebras (over $\Q$) and their associated eigenvarieties. We include two further short sections with complementary material. In Section \ref{sec:overview} we give a brief overview of different constructions of eigenvarieties appearing in the literature. In Section \ref{sec:occoh} we sketch the theory of overconvergent cohomology which generalizes the constructions of Section \ref{sec:defquat} --- we hope that providing this general context will illuminate the earlier material. 

\section{Overconvergent modular forms and the Coleman--Mazur eigencurve}\label{sec:ocmfs}
We'll begin by sketching some motivation for Coleman's theory of $p$-adic families of overconvergent modular forms. After that, we'll briefly describe the construction of the Coleman--Mazur eigencurve. These notes contain more details than the original lectures, but are still quite sketchy and the reader should refer to the original literature for precise details. In this section, we will make serious use of the geometry of modular curves. This will disappear in section \ref{sec:defquat}, where we'll talk about the (in many ways simpler) case of overconvergent automorphic forms for definite quaternion algebras.

\subsection{Local constancy of slopes}
The first main topic we'll discuss is Banach spaces of overconvergent modular forms. The Coleman--Mazur eigencurve is built using these spaces and the spectral theory of the compact Hecke operator, $U_p$. 

Fix a positive integer $N$ and a prime $p \nmid N$. For each integer $k$, we can consider the $\CC$-vector space of modular forms $M_{k,N} := M_k(\Gamma_1(N)\cap\Gamma_0(p))$ of level $\Gamma_1(N)\cap\Gamma_0(p)$ and weight $k$. Recall that $\Gamma_1(N)$ is the subgroup of $\SL_2(\Z)$ given by matrices which are strictly upper triangular mod $N$, whilst $\Gamma_0(p)$ is given by matrices which are upper triangular mod $p$. 

These spaces of modular forms come equipped with the Hecke operator $U_p$ which acts on $q$-expansions by 
\[U_p\left(\sum_n a_n q^n\right) = \sum_n a_{np}q^n.\]

Let $P_{k,N}(X)$ be the characteristic polynomial of $U_p$ on $M_{k,N}$. The polynomial $P_{k,N}$ actually lies in $\ZZ[X]$ --- we can see from the action on $q$-expansions that it preserves the $\Z$-lattice of modular forms with integral $q$-expansion. 

For convenience, we fix an algebraic closure $\Qpbar$ of $\Qp$ with additive valuation $v$ extending the $p$-adic valuation on $\Qp$. 

\begin{defn}
	For $h \in \Q_{\ge 0}$, let $d(k,h)$ denote the number of roots $\alpha$ of $P_{k,N}$ (counting multiplicities) with $v(\alpha) \le h$.  
\end{defn} 

The valuations of $U_p$ eigenvalues are traditionally called the \emph{slopes} of modular forms (since they are the slopes of the Newton polygon of $P_{k,N}$). The integer $d(k,h)$ is the degree of the slope $\le h$ factor of $P_{k,N}$.

\begin{theorem}[Hida, Coleman, Wan]\label{thm:locallyconstantslopes}
	Let $h \in \Q_{\ge 0}$. Let $k_1, k_2$ be two integers strictly bigger than $h+1$. There exists an integer $M_h$ (possibly depending on $p, N, h$, but not on the weights $k_i$) such that $k_1 \equiv k_2 \pmod{p^{M_h}(p-1)}$ implies $d(k_1,h) = d(k_2,h)$. 
\end{theorem}

\begin{itemize}
	\item Hida proved this for $h = 0$ --- in that case you can even take $M_h = 0$ ($p$ odd) or $M_h = 1$ ($p$ even).
	
	\item Wan showed that $M_h$ can be taken of order $O(h^2)$ (the implicit constants could depend on $p, N$).
	
	\item Before Coleman proved his result, Gouv\^{e}a and Mazur conjectured that this ``local constancy of slopes'' should hold with $M_h = \lceil h \rceil$ when the weights satisfy $k_i \ge 2h+2$. A counter-example to this was found by Buzzard and Calegari: $N = 1$, $p = 59$, $d(16, h) \ne d(3438,h)$ for some $0 < h \le 1$. Until recently, the slopes of modular forms were still quite mysterious --- however, conjectures due to Bergdall and Pollack \cite{BP1,BP2,BP3} and progress towards proving them in work of Liu, Truong, Xiao, and Zhao \cite{liu2023slopes} represent a big step forwards.
\end{itemize}

\subsection{Overconvergent modular forms}

\subsubsection{Modular curves}\label{sec:mc}Fix an integer $N \ge 5$ and a prime $p \nmid N$. We're going to consider two modular curves, with levels the congruence subgroups $\Gamma_1(N)$ and $\Gamma_1(N)\cap\Gamma_0(p)$ of $\SL_2(\Z)$. 

So, we have compact Riemann surfaces $X_\C := X(\Gamma_1(N)\cap\Gamma_0(p))$ and $Y_\C := X(\Gamma_1(N))$ containing the open modular curves $(\Gamma_1(N)\cap\Gamma_0(p))\backslash\cH$ and $\Gamma_1(N)\backslash\cH$. Modular forms of weight $k \in \Z$ and level $\Gamma$ can be identified with the space of sections $H^0(X(\Gamma),\omega^k)$ of the $k$th power of the modular line bundle $\omega$.

To investigate $p$-adic properties of modular forms, we'll need to look at an algebraic model of the Riemann surfaces $X_\C, Y_\C$. We have proper flat curves $X_{\Z[1/N]}$, $Y_{\Z[1/N]}$ over $\Spec(\Z[1/N])$ (cf.~Deligne--Rapoport \cite{del-rap}) providing integral models of $X_\C$ and $Y_\C$, with $R$-points (for a $\Z[1/N]$-algebra $R$) defined by:
\begin{align*}X_{\Z[1/N]}(R) &= \{\text{isomorphism classes of triples } (E/\Spec(R),H,\eta_N)\}\\
	Y_{\Z[1/N]}(R) &= \{\text{isomorphism classes of pairs } (E/\Spec(R),\eta_N)\}	
\end{align*} where $E/\Spec(R)$ is a (generalized\footnote{cf.~\cite{del-rap} or Conrad's article \cite{conrad-moduli} for details; note that a generalized elliptic curve might be singular, but it has a group law on its smooth locus, including an identity section $e: \Spec(R) \to E^{\sm}$. Our curve has level $\Gamma_1(N;p)$ in Conrad's notation.}) elliptic curve, $H \subset E[p]$ is a finite locally free rank $p$ subgroup scheme\footnote{When $E$ is only a generalized elliptic curve, we should add the condition that $H$ meets every irreducible component in every geometric fibre of $E$.}, and $\eta_N: \Z/N\Z\hookrightarrow E[N]$ is given by a point of order exactly $N$. The level structure $\eta_N$ will usually be easy to keep track of in everything that follows, so we will in general suppress it from our notation.

The spaces of sections $H^0(Y_\Q,\omega^{k})$, $H^0(X_\Q,\omega^k)$ give us $\Q$-vector spaces whose extension of scalars to $\C$ can be identified with spaces of modular forms of level $\Gamma_1(N)$ and $\Gamma_1(N)\cap\Gamma_0(p)$ respectively. This is the starting point for the algebro-geometric theory of modular forms. 

Since we are interested in $p$-adic properties of modular forms, we will extend scalars to $\Z_p$ to define $X:= (X_{\Z[1/N]})_{\Zp}$, $Y:= (Y_{\Z[1/N]})_{\Zp}$. The modular line bundle $\omega$ naturally extends to these integral models --- it is given by $\omega := e^*\Omega^1_{(E^{\mathrm{univ}})^{\sm}/X}$, with $e$ the identity section. So away from the cusps it's the sheaf of invariant differentials for the universal elliptic curve. The curve $Y$ with level prime to $p$ is smooth over $\Zp$.

\newpage
The (geometric) special fibres $X_{\Fpbar} \to Y_{\Fpbar}$ are sketched in  the following classic diagram:
\begin{figure}[h]\label{figure}
\includegraphics[width=0.8\textwidth]{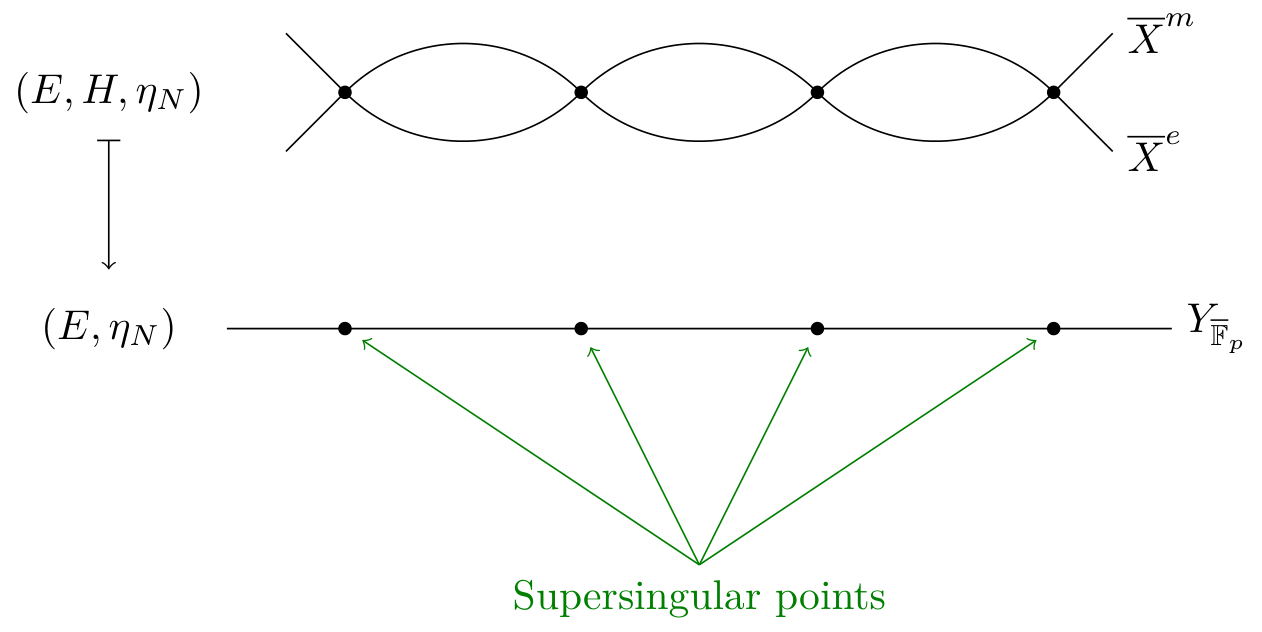}
\end{figure}

There are two open subschemes ${X}^m$, $X^e$ of $X_{\Fpbar}$ corresponding to the locus where $H$ is (\'{e}tale-locally) isomorphic to $\mu_p$ or $\Z/p\Z$ respectively. The letters $m$ and $e$ stand for multiplicative and \'{e}tale, respectively. The closures  $\overline{X}^m$, $\overline{X}^e$ are the irreducible components of $X_{\Fpbar}$ and their intersections correspond to points where the elliptic curve $E/\Fpbar$ is supersingular. The forgetful map $(E,H) \mapsto E$ induces an isomorphism from the component $\overline{X}^m$ to $Y_{\Fpbar}$, and a purely inseparable map of degree $p$ from the component $\overline{X}^e$ to $Y_{\Fpbar}$. On the other hand, the map $(E,H)\mapsto E/H$ induces an isomorphism from the component $\overline{X}^e$ to $Y_{\Fpbar}$, and a purely inseparable map of degree $p$ from the component $\overline{X}^m$ to $Y_{\Fpbar}$.

\subsubsection{Overconvergent modular forms} Our spaces of overconvergent modular forms of weight $k \in \Z$ (first defined by Katz \cite{katz-padic}) will be sections of the line bundle $\omega^k$ on certain affinoid subspaces of the adic generic fibre $\cX := (X)^{\ad}_{\Qp}$.

Since $X$ is proper over $\Zp$ we can construct its adic generic fibre by first considering the formal completion $\frakX$ of $X$ along its special fibre $X_{\Fp}$, and then $\cX = (\frakX)^{\ad}_{\Qp}$. See \cite[\S1.8]{kh-heidelberg}.

More explicitly, we can cover $\frakX$ by affine formal schemes $\Spf(R)$ (each $R$ is topologically of finite type over $\Spf(\Zp)$, in other words a quotient of an integral Tate algebra $\Zp\langle x_1,\ldots,x_n \rangle$). Then $(\frakX)^{\ad}$ is covered by $\Spa(R,R)$ and  $(\frakX)^{\ad}_{\Qp}$ is covered by $\Spa\left(R\left[1/p\right], {R}\right)$\footnote{In general we might need to take an integral closure of $R$ to get a ring of integral elements, but this is not necessary in our case since $\frakX$ is normal.}. 

There is a natural map of topological spaces $\mathrm{sp}: |\cX| \to |\frakX| = |X_{\Fp}|$, the specialization map. On affinoids, we take a point $x \in \Spa\left(R\left[1/p\right], {R}\right)$ to the open prime ideal $\{r \in R: |r|_x < 1\}\subset R$. We have the same story for $Y$, denoting the formal scheme and adic generic fibre by $\mathfrak{Y}$ and $\cY$.

\begin{exercise}
	Using smoothness of $Y$, show that if $y$ is a closed point of $Y_{\Fp}$, then $\mathrm{sp}^{-1}(y) \subset |\cY|$ is a closed subset which can be identified with the closure of an open unit disc (after extending scalars to $K/\Qp$, so $y$ is defined over the residue field of $K$). 
\end{exercise}

In $Y_{\Fp}$, we have an affine open subset $Y_{\Fp}^{\ord}$ corresponding to \emph{ordinary} elliptic curves (see Exercise \ref{ex:ampleaffine}). The complement is the supersingular locus, $Y_{\Fp}^{\mathrm{ss}}$, which consists of finitely many closed points (their field of definition contains $\FF_{p^2}$). The \emph{ordinary locus} $\cY^{\ord}$ in $\cY$ is the open affinoid $\mathrm{sp}^{-1}(Y_{\Fp}^{\ord})$. For a complete extension $L/\Qp$, $\cY^{\ord}(L) \subset Y(\cO_L)$ is given by the cusps together with elliptic curves $E/\cO_L$ with ordinary reduction.

Over $Y_{\Fp}^{\ord}$, the natural map $X_{\Fp}\to Y_{\Fp}$ (forget the subgroup $H$) has a canonical section, given by taking $H$ to be the connected part $E[p]^\circ$. This extends to a section on formal completions, and hence a map $\mathrm{can}:\cY^{\ord} \to \cX$. This is an isomorphism onto its image, which we denote by $\cX^{m,\ord}$ (over this locus $H=E[p]^\circ$ is \'{e}tale-locally isomorphic to the multiplicative group $\mu_p$).

If $C/\Qp$ is a complete and algebraically closed extension, and $E/\cO_{C}$ is an elliptic curve with ordinary reduction, the map `$\mathrm{can}$' is given by \[\mathrm{can}(E,\eta_N) = (E,H_{\mathrm{can}},\eta_N),\] with $H_{\mathrm{can}}$ the \emph{canonical subgroup} $\ker(E[p](\cO_C) \to E[p](\cO_C/p))$ coming from the $p$-torsion in the formal group $\widehat{E} \cong \widehat{\mathbb{G}}_m$. In particular, we have $H_{\mathrm{can}} \cong \mu_p$.

Overconvergent modular forms will be defined over affinoids $\cY^{\le v}$ (strictly) containing $\cY^{\ord}$. They can be defined using the Hasse invariant. This is a mod $p$ modular form $A \in H^0(Y_{\Fp},\omega^{p-1})$ which vanishes to order $1$ at the supersingular points and is non-vanishing on $Y_{\Fp}^{\ord}$. Its $q$-expansion is $1$.

\begin{exercise}\label{ex:ampleaffine}
	Use the fact that $\omega^{p-1}$ is an ample line bundle on $Y_{\Fp}$ and $Y_{\Fp}^{\ord}$ is the non-vanishing locus of a section to deduce that $Y_{\Fp}^{\ord}$ is affine.
\end{exercise}

We will use $A$ to define a `measure of supersingularity' on points of the generic fibre $\cY$, which will allow us to define the affinoids $\cY^{\le v}$. This seems to be explained slightly differently every time it shows up in the literature! Here is one way to do it. Cover $Y_{\Fp}$ by affine opens $U_i$ with trivialisations $s_i: \cO_{U_i}\cong \omega^{p-1}_{Y_{\Fp}}$ of the line bundle $\omega^{p-1}$. The affines $U_i$ match up with affine formal subschemes $\mathfrak{U}_i \subset \mathfrak{Y}$ and affinoids $\cU_i \subset \cY$. We can choose lifts of each trivialization $s_i$ to $\tilde{s}_i: \cO_{\cU_i} \cong \omega^{p-1}|_{\cU_i}$. We can also lift $A$ to a section $\widetilde{A}_i \in H^0(\mathfrak{U}_i,\omega^{p-1})$ for each $i$, and hence a function $f_i = \tilde{s}_i^{-1}(\widetilde{A}_i) \in \cO^+_{\cY}(\cU_i)$. 

\begin{defn}
	For $v \in [0,1)\cap \Q$, define $\cY^{\le v}$ to be the union of the affinoid subdomains $\cU_i^{\le v} \subset \cU_i$ defined by $|f_i|\ge |p|^v$.  
	
	In short, $\cY^{\le{v}}$ is the locus $|A| \ge |p|^v$, or $v_p(A) \le v$. 
\end{defn}
Take a look at Figure 3 in \cite{calegari-aws} to see a nice illustration of what's going on here. 

\begin{exercise}\label{exo:affinoids}
	Convince yourself of some of the following things:
	\begin{enumerate}
		\item The open subspace $\cY^{\le v} \subset \cY$ is independent of the various choices.
		\item Let $L/\Qp$ be a complete extension and let $x:\Spa(L,\cO_L) \to \cY$ be a point with corresponding elliptic curve $E_x/\cO_L$. The Hasse invariant defines an element $A_x$ in the free rank one $\cO_L/p$-module $(e^*\Omega^1_{E_{x,\cO_L/p}})^{\otimes p-1}$. If $A_x$ is non-zero, it has a well-defined valuation $v_p(A_x) \in [0,1)$ (compare $A_x$ with a basis vector). We have $x \in \cY^{\le v}$ if and only if $v_p(A_x) \le v$. 
		\item $\cY^{\le 0} = \cY^{\ord}$. 
		\item $\cY^{\le v}$ is obtained from $\cY$ by deleting (the closures of) smaller open discs inside each of the supersingular residue discs $\mathrm{sp}^{-1}(y)$ as $y$ runs over closed points of $Y_{\Fp}^{\mathrm{ss}}$ (at least after extending scalars to a field $K/\Qp$ whose residue field contains the field of definition of all the supersingular points).
		\item $\cY^{\le v}$ is affinoid. [One way to see this is to use the preceding item and \cite[Corollary 3.6a]{Coleman-RLC}. Another way is to adapt Exercise \ref{ex:ampleaffine} using a lift of a power of $A$ to a section of a very ample line bundle $\omega^{n(p-1)}$ on $Y$.] 
	\end{enumerate}
\end{exercise}

\begin{defn}
	For $v \in [0,1)\cap\Q$, the space of $v$-overconvergent modular forms of weight $k$ (and level $\Gamma_1(N)$) is defined to be
	\[M_k^{\dagger,v}(N) := H^0(\cY^{\le v},\omega^k).\]
	
	It comes with a natural (injective) map $H^0(Y_{\Qp},\omega^k) = H^0(\cY,\omega^k) \hookrightarrow M_k^{\dagger,v}(N)$ from the finite-dimensional space of classical modular forms, by restricting to $\cY^{\le v}$.
\end{defn}
Each $\Qp$-vector space $M_k^{\dagger,v}(N)$ is a finitely generated module over the affinoid algebra $\cO(\cY^{\le v})$, so we can make it into a Banach module over this ring and hence a $\Qp$-Banach space. In fact, for $v$ sufficiently small (but non-zero), Coleman \cite{Coleman-Banach} used a weight one Eisenstein series of level $\Gamma_1(p)$ to define a trivialisation of $\omega$ on $\cY^{\le v}$. This defines an isomorphism $M_k^{\dagger,v}(N) \cong \cO(\cY^{\le v})$.

We can equip the spaces $M_k^{\dagger,v}(N)$ with the usual Hecke operators $T_l, S_l, l\nmid Np$, and $U_l, l|N$. To describe the Hecke operator $U_p$ (most important for us), it is more natural to work with the modular curve of level $\Gamma_1(N)\cap\Gamma_0(p)$ instead of level $\Gamma_1(N)$.  

\subsubsection{$U_p$ and the canonical subgroup}\label{ssec:Upcan}
Our next task is to equip $M_k^{\dagger,v}(N)$ with a compact Hecke operator, $U_p$. This notation is usually reserved for Hecke operators at primes dividing the level of a modular form, so it is perhaps surprising that we will define an operator $U_p$ on a space of modular forms defined using $Y$ (which has level coprime to $p$). However, we can instead view $M_k^{\dagger,N}$ as a space of sections on an affinoid subspace of $\cX$ (the modular curve of level $\Gamma_1(N)\cap\Gamma_0(p)$). 

Recall that we have already defined a section $\mathrm{can}: \cY^{\ord} \to \cX$ to the natural forgetful map $\cX \to \cY$. 

\begin{theorem}[Katz, Lubin]\label{thm:cangp}
	Suppose $v < \frac{p}{p+1}$. The section $\mathrm{can}: \cY^{\ord} \to \cX$ extends to $\cY^{\le v}$, and gives an isomorphism onto its image which we denote by $\cX^{m,\le v}$. 
	
	On $L$-points this sends $(E,\eta_N)$ to $(E,H_{\mathrm{can}}(E),\eta_N)$, where $H_{\mathrm{can}}(E)$ is, by definition, the \emph{canonical subgroup} of $E[p]$. 
	
	If $C \subset E[p]$ is a rank $p$ locally free subgroup scheme which is \emph{not} the canonical subgroup, and $(E,\eta_N) \in \cY^{\le v}$ for $v < \frac{p}{p+1}$, then $(E/C,\eta_N \bmod C) \in \cY^{\le v/p}$ and $H_{\mathrm{can}}(E/C) = E[p]/C$. 
\end{theorem}

The affinoid $\cX^{m, \le v} \subset \cX$ can also be described intrinsically as the locus where the Fargues degree of $H$ satisfies a certain inequality, see \cite{BP}. 

Using the isomorphism $\cY^{\le v} \cong \cX^{m,\le v}$, we can now identify \[M_k^{\dagger,v}(N) = H^0(\cX^{m,\le v},\omega^{k}).\] In particular, we see that it has a natural injective restriction map from a space of classical modular forms
$H^0(X_{\Qp},\omega^{k-1})$ with level $\Gamma_1(N)\cap\Gamma_0(p)$. We will now define $U_p$ on $M_k^{\dagger,v}(N)$ in a way that's compatible with the usual definition on $H^0(X_{\Qp},\omega^{k})$. 

From now on we will abbreviate our notation, and use $\cX^{\le v}$ to denote $\cX^{m,\le v}$ (which is isomorphic with $\cY^{\le v}$). Note that $\cX^{\le v}$ is \emph{not} the inverse image of $\cY^{\le v}$ in $\cX$. The full inverse image also contains (a neighbourhood of) the locus where $H$ lifts $E_{\Fpbar}[p]^{\et}$. 

Working geometrically, we think about the following correspondence on $X_{\Qp}$:

\begin{equation}\label{Upcorr}\xymatrix{& Z_{\Qp} \ar[dl]_{p_1}\ar[dr]^{p_2} \\ X_{\Qp}& & X_{\Qp}} \xymatrix{& (E,H,C) \ar[dl]\ar[dr] \\ (E,H)& & (E/C,E[p]/C)}  \end{equation}

where $Z_{\Qp}$ is a moduli space of quadruples $(E,H,C,\eta_N)$ where the triple $(E,H,\eta_N)$ gives a point of $X_{\Qp}$ and $C \subset E[p]$ is another rank $p$ subgroup scheme with $C\cap H = \{e\}$. Pulling back differentials from $E/C$ to $E$ defines an isomorphism of line bundles $p_2^*\omega^k \cong p_1^*\omega^k$. There is a natural adjunction map $p_{1,*}p_1^*\omega^k \to \omega^k$ which is a trace map on sections ($p_1$ is finite flat). See, for example, \cite[\S 5]{BuzWild} for more details.

The Hecke operator $U_p$ on $H^0(X_{\Qp},\omega^{k})$ is obtained as the following composition: 

\begin{align*}H^0(X_{\Qp},\omega^{k}) \xrightarrow{p_2^*} H^0(Z_{\Qp},p_2^*\omega^{k}) \cong H^0(Z_{\Qp},p_1^*\omega^{k}) = H^0(X_{\Qp},p_{1,*}p_1^*\omega^{k}) \\\xrightarrow{\frac{1}{p}\mathrm{tr}_{p_1}} H^0(X_{\Qp},\omega^{k})\end{align*}

Note the rescaling by $1/p$. When $k \ge 2$, we have $\omega^k \cong \omega^{k-2}\otimes\Omega^1_{X_{\Qp}}(\textrm{Cusps})$ (allow simple poles at the cusps), and $U_p$ can be described using the trace map on differentials of $X$, with no rescaling.
\begin{exercise}
	Check the $U_p$ operator defined above coincides with the usual Hecke operator in the theory of modular forms. On $q$-expensions it acts as \[U_p: \sum_n a_nq^n \mapsto \sum_n a_{np}a^n.\]
\end{exercise}

Now we are ready to restrict our Hecke operator to overconvergent affinoids. Theorem \ref{thm:cangp} is absolutely crucial here. It shows that for $v < \frac{p}{p+1}$ our diagram (\ref{Upcorr}) restricts to 

\begin{equation}\label{Upcorr:oc}\xymatrix{&p_1^{-1}(\cX^{\le v}) \ar[dl]_{p_1}\ar[dr]^{p_2} \\ \cX^{\le v}& & \cX^{\le v/p} }\end{equation}

We can therefore define an operator $U_p: H^0(\cX^{\le v/p},\omega^{k}) \to H^0(\cX^{\le v},\omega^{k})$, compatible with $U_p$ on the classical subspaces $H^0(X_{\Qp},\omega^{k})$. Note that $U_p$ increases the radius of overconvergence (if $v > 0$).

By abuse of notation, we also use $U_p$ to denote the composition:
\[H^0(\cX^{\le v},\omega^{k})  \to  H^0(\cX^{\le v/p},\omega^{k})  \xrightarrow{U_p} H^0(\cX^{\le v},\omega^{k})\] with the first map given by restriction. When $v < \frac{1}{p+1}$ we can alternatively describe this map as the composition: 

\[H^0(\cX^{\le v},\omega^{k})  \xrightarrow{U_p}  H^0(\cX^{\le pv},\omega^{k})  \rightarrow H^0(\cX^{\le v},\omega^{k}).\]

\begin{prop} Suppose $0< v < \frac{p}{p+1}$. The endomorphism $U_p$ is a compact operator on the $\Qp$-Banach space $H^0(\cX^{\le v},\omega^{k})$.
\end{prop} 
\begin{proof}
	$U_p$ factorises through the restriction map \[H^0(\cX^{\le v},\omega^{k}) \to  H^0(\cX^{\le v/p},\omega^{k}),\] so it suffices to prove that this is compact. We can rewrite this map as \[H^0(\cX^{\le v},\omega^{k}) \xrightarrow{1\otimes\id} \cO(\cX^{\le v/p})\widehat{\otimes}_{\cO(\cX^{\le v})}H^0(\cX^{\le v},\omega^{k})\] and $H^0(\cX^{\le v},\omega^{k})$ is a finite projective $\cO(\cX^{\le v})$-module. So in fact it suffices to show that the restriction map on affinoid algebras $\cO(\cX^{\le v}) \rightarrow \cO(\cX^{\le v/p})$ is compact. This is analogous to the fact that restricting functions from a closed disc to a sub-disc of smaller radius is a compact map. If you followed the second hint in Exercise \ref{exo:affinoids}(5), you might have described the affinoids $\cX^{\le v}$ as intersections of $\cY$ with explicit affinoid subspaces of the projective space $\mathbb{P}(H^0(Y_{\Qp},\omega^{n(p-1)}))$ for $n$ sufficiently large. With this description, the compactness of the restriction maps should be quite clear.
\end{proof}
We let $\TT^N$ be the `abstract Hecke algebra', a polynomial algebra over $\Qp$ generated by $\{U_p,T_l,S_l: l \nmid Np\}$ (we omit the Hecke operators at primes dividing $N$ for simplicity; you might want to include them if you want to capture the full $q$-expansion of an eigenform). This acts on the spaces $H^0(\cX^{\le v},\omega^{k})$.

Now we have the spectral theory explained in \cite{jl-heidelberg} at our disposal. In particular, we have a characteristic power series:

\[F_{k,N}^{\dagger,v} := \det(1-XU_p|M_k^{\dagger,v}(N)) \in \Qp\left\{\left\{X \right\}\right\}.\] 

Here, $\Qp\left\{\left\{X \right\}\right\}$ is the ring of entire power series, as defined in \cite{jl-heidelberg}.

Comparing the two factorizations of $U_p$ and applying \cite[Lemma 2.38]{jl-heidelberg} tells us that $F_{k,N}^{\dagger,v} = F_{k,N}^{\dagger,v/p}$, and so $F_{k,N}^{\dagger,v}$ is actually independent of $v \in (0,\frac{p}{p+1})$. We therefore drop the $v$ from the notation. The characteristic power series $F^\dagger_{k,N}(X)$ has the polynomial $X^{\deg(P_{k,N})}P_{k,N}(1/X)$ as a factor, corresponding to the classical subspace $H^0(X_{\Qp},\omega^{p-1}) \subset M_k^{\dagger,v}(N)$. 

Any polynomial factor $Q$ of $F_{k,N}^{\dagger}$ defines a finite dimensional subspace \[(M_k^{\dagger,v}(N))^Q \subset M_k^{\dagger,v}\] which is independent of $v$. We can for example choose $Q$ to be the factor corresponding to the roots $\alpha$ of $F_{k,N}^{\dagger}$ with $v_p(\alpha)\ge -h$. Then we define $M_k^{\dagger}(N)^{\le h}$ to be the corresponding direct summand. It is spanned by generalized $U_p$-eigenvectors with $U_p$-eigenvalue the inverse of a root of $Q$. 

Considering $q$-expansions, we see that all the slopes of $U_p$ are $\ge 0$. An important theorem of Coleman tells us that overconvergent modular forms with small (compared to $k$) slope are classical:

\begin{theorem}[Coleman, \cite{coleman-classicality1}]\label{thm coleman-classical}
	If $h < k - 1$, the map $H^0(X_{\Qp},\omega^{k})^{\le h} \subset M_k^{\dagger}(N)^{\le h}$ is an isomorphism.
\end{theorem}
Apart from the original proof, alternative arguments which generalise well have been given by Kassaei \cite{Kassaei} and Boxer--Pilloni \cite{BP}. The latter argument involves identifying the cokernel of the inclusion from the classical subspace as a relative cohomology group (i.e~cohomology with support in a closed subset) which also admits a compact $U_p$-operator, and showing that the slopes of $U_p$ on this cokernel are all $\ge k-1$. 

We refer to Calegari's AWS notes \cite{calegari-aws} as somewhere to read lots more about overconvergent modular forms. 

\subsection{Weight space}\label{sec wtspace}
We will soon discuss families of modular forms, so we want to define Banach modules over affinoid algebras, whose fibres at closed points will give us back the spaces $M_k^{\dagger,v}(N)$. Our affinoids will be opens in \emph{weight space} which we discuss now. 

\begin{defn}
	Formal weight space $\mathfrak{W}$ is defined to be the adic space \[\mathfrak{W}:=\Spa(\Zp\lb\Zp^\times\rb,\Zp\lb\Zp^\times\rb).\] Its maps from an affinoid $\Spa(R,R^+)$ correspond to continuous homomorphisms $\Zp^\times \to R^\times$. \end{defn}
	
	The weight space we will mostly be interested in is the generic fibre $\cW:= (\mathfrak{W})_{\Qp}$. The maps from affinoid adic spaces $\Spa(R,R^+)$ over $\Spa(\Qp,\Zp)$ to $\cW$ correspond to continuous homomorphisms $\Zpx \to R^\times$.

When $(R,R^+)$ is a $\Qp$-affinoid, we will refer to a continuous homomorphism $\Zpx \to R^\times$ as an \emph{affinoid weight}. We will always assume $R$ is Noetherian. 

We can make $\cW$ much more explicit by thinking about $\Zpx$. For notational simplicity, we now assume $p$ is odd. So $\Zpx \cong (\Z/p\Z)^\times \times (1+p\Zp)$. Let $\Delta$ denote the finite group of characters $(\Z/p\Z)^\times \to \Zpx$. If $\kappa: \Zpx \to R^\times$ is a continuous homomorphism, we set $\epsilon(\kappa):= \kappa|_{(\Z/p\Z)^\times} \in \Delta$. 

\begin{exercise}\label{exo:wtspace}
	Show that the map $\kappa \mapsto (\epsilon(\kappa),\kappa(1+p)-1)$ gives an isomorphism between $\cW$ and $\Delta \times D(0,<1)$, where $D(0,<1)$ is the open unit disc over $\Qp$. 
	
	(Two observations that might be helpful are that \[D(0,<1) = \Spa(\Zp\lb t\rb,\Zp\lb t\rb)_{\Qp}\] and that maps $\Spa(R,R^+) \to D(0,<1)$ biject with the subset of topologically nilpotent elements $R^{\circ\circ}$.)
\end{exercise}

Something we will use when we talk about definite quaternion algebras is that continuous homomorphisms $\kappa: \Zpx \to R^\times$ to $\Qp$-affinoids are necessarily \emph{locally analytic}. This means that there exists an $n \ge 1$ such that for each $a \in (\Z/p^n\Z)^\times$, the function $\kappa: (a+ p^n\Zp) \to R^\times$ is given by evaluating a power series $\kappa(a+p^n z) = f(z) \in R\langle z \rangle$. This can be proved using the $p$-adic $\log$ and $\exp$ (see \cite[Theorem 6.3.4]{eigenbook}).

Of course we want to interpret the usual integral weights $k$ of our modular forms as points in weight space:

\begin{defn}
	For $k \in \Z$ we define an \emph{integral} weight $[k]: \Zp^\times \to \Qpx$ by $[k](z) = z^{k-2}$. 
\end{defn}

The normalization of sending weight $2$ to the identity map is because when we think about modular forms in terms of Betti cohomology, weight $2$ corresponds to the trivial local system. 

For the statement of the next lemma, recall that a subset of an adic space is `Zariski dense' if the smallest closed subset, defined by a coherent ideal sheaf, containing this subset is the whole space. See, for example, \cite[\S2.1]{JN-interp} for more details.

\begin{lemma}
	The integral weights are a Zariski dense subset of $\cW$. Moreover, each integral weight $[k]$ has a basis of affinoid neighbourhoods $V$ in $\cW$ such that the integral weights in $V$ are Zariski dense.
\end{lemma}
\begin{proof}
	Using Exercise \ref{exo:wtspace}, we can reduce to a statement about points of the form $(1+p)^k-1$ in the open unit disc. If we take any sequence of integers $k_m$ with $k_m \equiv k \mod{(p-1)p^m}$, then the sequence $(1+p)^{k_m}-1$ accumulates at $(1+p)^k-1$. In particular, there are no non-zero functions on an affinoid neighbourhood of $(1+p)^k-1$ which vanish on the sequence $((1+p)^{k_m}-1)_{m \ge 1}$.
\end{proof}

The second part of the Lemma is often important in practice, when one wants to work in a small neighbourhood of an integral weight.

\subsection{Families of overconvergent modular forms and proving local constancy of slopes}\label{ssec:locconstslope}
Another key contribution of Coleman \cite{Coleman-Banach} was to define for each affinoid weight $\kappa: \Zpx \to R^\times$, and $v$ sufficiently small (depending on $\kappa$), a Banach $R$-module $M_\kappa^{\dagger,v}(N)$ of overconvergent modular forms of weight $\kappa$. We think of this module as parameterizing families of overconvergent modular forms over $\Spa(R,R^+)$.  

There is a compact operator $U_p$ on $M_\kappa^{\dagger,v}(N)$, together with the other Hecke operators, such that whenever we have $f_k: R \to \Qp$ with $f_k\circ\kappa = [k]$, we have a Hecke-equivariant isomorphism

\[f_k^*:M_\kappa^{\dagger,v}(N)\widehat{\otimes}_{R,f_k^*}\Qp \cong M_k^{\dagger, v}(N).\]

To prove local constancy of slopes, we consider the following: fix a weight $k_1 \in \Z$, and a slope $h \in \Q_{\ge 0}$. The characteristic power series $F_{k_1,N}^{\dagger}$ has a factorisation $F_{k_1}^{\le h}F_{k_1}^{> h}$, where the Newton polygons of the two factors have all slopes $\le h$ and $> h$ respectively. The result \cite[Proposition 3.19]{jl-heidelberg} shows that this factorisation extends to an affinoid  neighbourhood $U$ of $[k_1]$ in $\cW$ (which we assume to be connected), with corresponding affinoid weight $\kappa$.

In particular, we can find a finite projective $\cO(U)$-module $M_\kappa^{\dagger,v}(N)^{\le h}$, a direct summand of $M_\kappa^{\dagger,v}(N)$, such that $f_k^*$ induces isomorphisms \[M_\kappa^{\dagger,v}(N)^{\le h}{\otimes}_{R,f_k^*}\Qp \cong M_k^{\dagger, v}(N)^{\le h}\] for every $[k] \in U$. This proves that the vector space on the right hand side has the same dimension whenever $[k] \in U$ --- the rank of the finite projective module. In particular, if $h < k - 1$, Coleman's classicality theorem shows that this coincides with the dimension of the appropriate `small slope' space of \emph{classical} modular forms. That completes a sketch of the proof of Theorem \ref{thm:locallyconstantslopes}.

We won't go into how Coleman defined families of overconvergent modular forms --- it involved twisting by a $p$-adic family of Eisenstein series, see \cite[\S6]{Buz07} for details. It is also explained there why the Banach modules $M_\kappa^{\dagger,v}$ satisfy `property (Pr)' \cite[Definition 2.39]{jl-heidelberg}.

A modern alternative to Coleman's approach was initiated by Andreatta--Iovita--Stevens \cite{ais-hilbert}, Pilloni \cite{pilloni-oc} and Chojecki--Hansen--Johansson \cite{CHJ}, using the Hodge--Tate period map to give geometric definitions for families of overconvergent modular forms. 

\subsection{The Coleman--Mazur eigencurve}

We have now introduced everything needed to construct the Coleman--Mazur eigencurve \cite{Col98}, using the `eigenvariety machine' presented in \cite{jl-heidelberg}. Coleman and Mazur originally worked with tame level $N = 1$, and $p > 2$, but these restrictions were removed by Buzzard \cite{Buz07}. 

As we recalled in the previous section, for affinoid opens $U \subset \cW$ and sufficiently small $v$, we have Banach modules of overconvergent modular forms $M_{\kappa_U}^{\dagger,v}(N)$ equipped with a compact operator $U_p$. Here $\kappa_U: \Zpx \to \cO(U)^\times$ is the affinoid weight induced by the embedding $U \subset \cW$.

 The characteristic power series of $U_p$ on these modules is independent of $v$ (as long as $v$ is sufficiently small), as in the fixed weight case this is proved using \cite[Lemma 2.38]{jl-heidelberg}. So they glue to give a characteristic power series $F^\dagger_N \in \cO(\cW)\{\{X\}\}$ which specializes to $F_{k,N}^{\dagger}$ at each integral weight $[k]$. In fact, since all the slopes of $U_p$ are $\ge 0$\footnote{This can be checked by considering the effect of $U_p$ on $q$-expansions of overconvergent modular forms.}, this power series is in $\cO^+(\cW)\{\{X\}\} = \Zp\lb\Zpx\rb\{\{X\}\}$.

As explained in \cite[\S5.2]{jl-heidelberg}, this allows us to construct the \emph{spectral curve} $\cZ = V(F_N^\dagger)$, the Fredholm hypersurface associated to $F^\dagger_N$. The `Riesz theory' developed in \cite[Theorem 4.6]{jl-heidelberg} takes the Banach modules of overconvergent modular forms and glues them into a coherent sheaf $\cN$ over $\cZ$ with an action of the Hecke algebra $\psi: \TT^N\to\End(\cN)$. Then the eigenvariety datum $(\cZ,\cN,\TT,\psi)$ gives rise to an adic space $\cE = \cE_N \xrightarrow{\pi} \cZ$, finite over the spectral curve. The maximal points $x \in \Max(\cE)$ correspond to systems of eigenvalues for the Hecke algebra $\TT^N$ in the spaces $M_{w(x)}^{\dagger,v}(N)$.\footnote{We have chosen to construct our eigencurve without Hecke operators at places dividing $N$. This is usually the simplest option. If you want to pin down the $q$-expansion of an overconvergent eigenform precisely, then you ned to include operators $U_l$ for $l|N$.} This is the \emph{tame level} $N$, $p$-adic eigencurve. It comes with a map to weight space $w: \cE \to \cW$ (the composite of $\pi$ with the map from the spectral curve to weight space).

The space $\cE$ comes equipped with its set of \emph{classical points} $Z$ which are the points $x \in \Max(\cE)$ with $w(x) = [k]$ for some integer $k$ and whose corresponding system of Hecke eigenvalues appears in $H^0(X_{\Qp},\omega^{k})$.\footnote{We should also include more general classical points coming from modular forms with level $\Gamma_1(p^n)$ and non-zero $U_p$-eigenvalue, but we have restricted to level $\Gamma_0(p)$ for simplicity.} These systems of Hecke eigenvalues come from classical modular forms of level $\Gamma_0(p)\cap\Gamma_1(N)$ and weight $k$.

Here are some important properties of $\cE$. We will sketch the proofs of these properties in the next section:

\begin{enumerate}
	\item The map $w: \cE\to \cW$ is locally quasi-finite and flat. In particular, $\cE$ is equidimensional of dimension one.
	\item For any irreducible component $\cC \subset \cE$, $w(\cC)$ is the complement of finitely many points in an irreducible component of $\cW$.
	\item\label{dense} The subset $Z \subset \Max(\cE)$ is Zariski 
	dense and accumulates at itself (in other words, each point of 
	$Z$ admits a basis of affinoid neighbourhoods $V$ such that $V \cap Z$ 
	is Zariski dense in $V$).
	\item\label{reduced} $\cE$ is reduced.
	\item $\cE$ carries a family of Galois (pseudo-)representations: there is a continuous function $T : \Gal(\Qbar/\Q) \to \cO^+(\cE)$ such that for each classical point $x \in Z$, the specialization $T_x$ is the trace of the Galois representation associated to a classical modular form whose Hecke eigenvalues match up with $x$.\footnote{It was shown by Coleman--Mazur that this family of pseudo-representations can be lifted (non-canonically) to a Galois representation on a vector bundle over the eigencurve. See for example \cite[Prop.~3.1.1, Remark~3.1.2]{JNwt2}.}
\end{enumerate}

\subsection{Geometric properties of the eigencurve}\label{sec: geom props}
As in the previous section, $\cE$ denotes the tame level $N$, $p$-adic eigencurve, with its map $w: \cE\to\cW$ to weight space. Our first task is to explain why the `eigencurve' is in fact a curve:

\begin{prop}\label{prop: ec equidim} The map $w: \cE\to \cW$ is locally quasi-finite and flat. In particular, $\cE$ is equidimensional of dimension one.
\end{prop}
\begin{proof}
	Recall that $\cE$ is finite over the spectral curve $\cZ$. It follows from general properties of spectral varieties \cite[Theorem  3.9]{jl-heidelberg} that $\cZ$ is locally quasi-finite and flat over $\cW$, so $\cE$ is locally quasi-finite over $\cW$. To prove the flatness of $\cE$ over $\cW$, we make use of the fact that our $\cW$ is one-dimensional. Recall that, by the key construction of \cite[\S5]{jl-heidelberg}, $\cE$ has an open cover by affinoids of the form \[w: \Spa(R,R^+) \to U,\] where $U$ is a connected affinoid open in $\cW$ and $R$ is a subring of $\End_{\cO(U)}(P)$ for a finite projective $\cO(U)$-module $P$.
	
	Since $\cW$ is smooth and one-dimensional, $\cO(U)$ is a Dedekind domain. Since $R$ is torsion-free over $\cO(U)$, it is flat. 
\end{proof}

Perhaps the most important property of the eigencurve is the density of classical points. It is of crucial importance when we use information about classical modular forms to establish properties of the eigencurve (cf.~the proof that the eigencurve is reduced, Proposition \ref{prop:reduced}). As a prelude to this, we need to show that the eigencurve contains plenty of points with integral weight. 
\begin{prop}\label{prop: image of irred comp in weight space}
	For any irreducible component $\cC \subset \cE$, $w(\cC)$ is the complement of finitely many points in an irreducible component of $\cW$.
\end{prop}
\begin{proof}
The map from $\cE$ to the spectral curve $\cZ = V(F_N^\dagger)$ is finite and $\cE$ is equidimensional of dimension equal to $\dim(\cZ)$, the image of an irreducible component $\cC$ in $\cZ$ is an irreducible component $\cC_\cZ$. The image $w(\cC)$ is contained in an irreducible component $\cW_{\chi}$ of $\cW$, corresponding to a character $\chi: (\Zpx)^{\mathrm{tors}} \to \Zpx$ of the torsion subgroup of $\Zpx$. As discussed in \S\ref{sec wtspace}, we have $\cW_{\chi}\cong \Spa(\Zp\lb t\rb,\Zp\lb t\rb)_{\Qp}$. The relevant piece of the spectral curve, $\cZ_\chi := \cZ|_{\cW_\chi}$ is the Fredholm hypersurface of a power series  $F_\chi\in \Zp\lb t\rb\{\{X\}\}$.

It follows from \cite[Corollary 4.2.3]{Con99} that the irreducible component $\cC_\cZ$ is itself a Fredholm hypersurface, given by the vanishing locus of an irreducible factor $F_{\cC}$ of $F_\chi$. Write $F_{\cC} = 1 + \sum_{n\ge 1}a_n(t) X^n \in \Zp\lb t\rb\{\{X\}\}$. The points of $\cW_\chi$ which are missing from $w(\cC)$ correspond to the common zeroes of the $a_n(t)$. By Weierstrass preparation, they are a finite set. 
\end{proof}
\begin{theorem}\label{thm classical dense}
	The subset $Z \subset \Max(\cE)$ of classical points is Zariski 
	dense and accumulates at itself (in other words, each point of 
	$Z$ admits a basis of affinoid neighbourhoods $V$ such that $V \cap Z$ 
	is Zariski dense in $V$).
\end{theorem}
\begin{proof}
Pick an irreducible component $\cC$ of $\cE$. By Proposition \ref{prop: image of irred comp in weight space}, we know that $w(\cC)$ contains integral weights and we choose one, $[k_0] \in w(\cC)$. Then choose a point $x \in w^{-1}([k_0])$. We have an affinoid open neighbourhood $x \in U \subset \cE$, with $w|_U$ finite flat and $w(U) = V$ a connected affinoid open neighbourhood of $[k_0]$. 

Consider the integral weights $[k]$ in $V$ with $k = k_0+(p-1)p^M$. The open neighbourhood $V$ contains all these weights for $M$ sufficiently large. By Coleman's classicality theorem (Theorem \ref{thm coleman-classical}), and the fact that function $U_p^{-1}$ on the affinoid $U$ is bounded, the fibre $w^{-1}([k])\cap U$ consists entirely of classical points when $M$ is sufficiently large. This gives a Zariski-dense family of integral weights $[k]$ in $V$, so the inverse image by the finite flat map $w$ is Zariski-dense in $U$ and hence in $C$\footnote{The finite flatness of the map tells us that the inverse image of a Zariski dense subset is Zariski dense. See \cite[Lemma 3.8.7]{eigenbook}. Actually, we don't need to use flatness of the map $w: \cE \to \cW$, it is enough to know that $w|_U$ maps irreducible components surjectively to $V$ \cite[Lemme 6.2.8]{chenevier-GLn}.}. The same argument gives the accumulation property for $Z$.
\end{proof}

\begin{rem}
	Note that the classical points are not dense for the usual topology on the adic space $\cE$. Moreover, integral weights are not dense in weight space for the usual topology. Identifying a component of weight space with the open unit disc, the integral weights each lie in one of the open affinoids $\{|t|_p = p^{-i}\}$ for $i \in \ZZ_{> 0}$. This means that we can easily write down an open subset of weight space which contains no integral weights --- for example, $\{ |t|^{-3/2}_p \le |t|_p \le p^{-4/3}\}$.
\end{rem}

With density of classical points proven, we can give an example of a geometric property of the eigencurve which is deduced from known properties of classical modular forms.

\begin{prop}\label{prop:reduced}
$\cE$ is reduced.
\end{prop}
\begin{proof} We sketch the proof of this, following Chenevier \cite{chenevier-JL} and Bella\"{i}che \cite[Theorem 3.8.8]{eigenbook}. The key ingredient is that the local ring $\cO_{\cE,z}$ is reduced for $z \in Z$ a classical point. It then follows from density of classical points and some commutative algebra \cite[Lemme 3.10, 3.11]{chenevier-JL} that $\cE$ is reduced.

We prove the claim that $\cE$ is reduced in a neighbourhood of each classical point. For concreteness, we will use the notion of slope data \cite[\S1.3.4]{jl-heidelberg}. Let $z \in \cE$ be a classical point. By \cite[Corollary 3.20]{jl-heidelberg}, the point $\pi(z)$ of the spectral curve has an open neighbourhood $V_{U,h}$ for $U \subset \cW$ connected affinoid open and  $h \in \Q_{\ge 0}$. Let $\widetilde{V} = \pi^{-1}(V_{U,h})$. Then $\widetilde{V}$ is an open neighbourhood of $z$ in $\cE$ with 
\[\cO(\widetilde{V}) = \im\left(\cO(U)\otimes_{\Qp}\TT^{N} \to \End_{\cO(U)}(M^{\dagger,v}_{\kappa_U}(N)^{\le h})\right).\] Consider the integral weights $[k] \in U$ with $h < (k-1)/2$. They are Zariski dense in $U$ (since $U$ already contains the integral weight $k_0 = w(z)$). The image of $\cO(\widetilde{V})$ in $\End_{\Qp}(M^{\dagger}_{k}(N)^{\le h})$ is identified with \[\im\left(\TT^{N} \to \End_{\Qp}(M^{\dagger}_{k}(N)^{\le h})\right) = \im\left(\TT^{N} \to \End_{\Qp}(H^0(X_{\Qp},\omega^{k})^{\le h})\right)\] by the classicality theorem. We claim that this (classical) Hecke algebra is semisimple. The Hecke operators $T_l, S_l$ for $l\nmid pN$ act semisimply on the whole space $M_{k,N}$ --- this is a standard fact which follows from the normality of these Hecke operators with respect to the Petersson inner product on the cuspidal subspace (the Eisenstein series can be handled separated). 	We also need to know that $U_p$ acts semi-simply. On the $p$-new (cuspidal) subspace of $M_{k,N}$, the $U_p$-operator is normal with respect to the Petersson inner product. From the automorphic representations point of view, this is a consequence of the fact that unramified twists of the Steinberg representation of $\GL_2(\Qp)$ have a one-dimensional space of invariants under the Iwahori subgroup. Then we need to worry about the space of $p$-old forms, but that is why we chose the slope bound $h < (k-1)/2$. This means that each Hecke eigenform of level $\Gamma_1(N)$ contributes at most a one-dimensional space to $H^0(X_{\Qp},\omega^{k})^{\le h}$, so $U_p$ is forced to act semisimply. Compare with the proof of \cite[Theorem 4.2]{coleman-edixhoven}.

	Now we claim that $\cO(\widetilde{V})$ is reduced. Suppose $t \in \cO(\widetilde{V})$ with $t^n = 0$. We can identify $t$ with an endomorphism of the finite projective $\cO(U)$-module $P := M^{\dagger,v}_{\kappa_U}(N)$, and we know that $t$ maps to zero in $\End(P\otimes_{\cO(V)}k(\m))$ for a dense set of maximal ideals $\m \in \Spec(\cO(U))$ --- i.e.~those $\m$ corresponding to integeral weights $[k] \in U$ with $h < (k-1)/2$.  Extending by zero to a free module with $P$ as a direct summand, we can view $t$ as an endomorphism of a finite free $\cO(U)$-module, or in other words as a matrix. We've shown that each matrix entry lies in the intersection of a Zariski dense set of maximal ideals. Since $\cO(U)$ is reduced, this implies that each matrix entry is $0$, and hence $t=0$.
\end{proof}

Finally, we discuss Galois representations. For each classical point $x \in \cE$, there is a continuous semisimple Galois representation 
\[\rho_x : \Gal(\Qbar/\Q) \to \GL_2(k(x)) \]
with coefficients in the residue field $k(x)$ of the point $x$, with the properties:
\begin{itemize}
	\item $\rho_x$ is unramified at primes $l\nmid Np$,
	\item for $l\nmid Np$, the characteristic polynomial of $\rho_x(\Frob_l)$ is $X^2 - T_l(x)X + lS_l(x)$.
\end{itemize}
By the Chebotarev density theorem, specifying the characteristic polynomials of Frobenius elements determines the semisimple representation $\rho_x$ uniquely up to isomorphism. To make sense of the second property, recall that $T_l$ and $S_l$ can be regarded as elements of the ring of global functions on $\cE$ (indeed, there is a canonical map $\TT^N \to \cO^+_{\cE}(\cE)$, coming from the construction of $\cE$).

Interpolating the representations $\rho_x$ into a family of representations over $\cE$ turns out to be a little bit subtle, essentially because a (generically irreducible) family of representations can specialize to a non-semisimple representation. Something which is easier to do (and still captures a lot of information), is to interpolate the functions
\[T_x := \mathrm{trace}(\rho_x): \Gal(\Qbar/\Q)\to k(x)^+.\] 

\begin{prop}\label{prop:pseudocharacter}
There is a continuous function
\[T: \Gal(\Qbar/\Q) \to \cO^+_{\cE}(\cE),\] which factors through $G_{\Q,Np}$, the Galois group of the maximal extension of $\Q$ unramified at primes $l \nmid Np$, and satisfies $T(\Frob_l) = T_l$ for $l \nmid Np$.

We also have $T(\id) = 2$, $T(g_1 g_2) = T(g_2 g_1)$ and the identity: 
\begin{multline*}T(g_1)T(g_2)T(g_3) - T(g_1)T(g_2g_3) - T(g_2)T(g_1g_3) - T(g_3)T(g_1g_2) \\ + T(g_1g_2g_3) + T(g_1g_2g_3) = 0\end{multline*} for all $g_1, g_2, g_3 \in \Gal(\Qbar/\Q)$.
\end{prop}
\begin{proof}
First we explain what topology we put on $\cO^+_{\cE}(\cE)$. We give it the subspace topology with respect to the following topology on $\cO_{\cE}(\cE)$: it is given the coarsest locally convex topology such that all the restriction maps $\cO_{\cE}(\cE) \to \cO_{\cE}(U)$, $U\subset \cE$ affinoid open, are continuous. Since $\cE$ is reduced, $\cO^+_{\cE}(\cE)$ is compact \cite[Lem.~7.2.11]{bellaiche_chenevier_pseudobook}. The functions $T_x$ give us a map $T: G_{\Q,Np} \to \prod_{z \in Z} k(x)^+$. The ring $\cO^+_{\cE}(\cE)$ injects into $\prod_{z \in Z} k(x)^+$, by Zariski-density of classical points. For each prime $l \nmid Np$, the image $T(\Frob_l)$ lands in the closed subring $\cO^+_{\cE}(\cE)$, so by Chebotarev density $T$ has image in $\cO^+_{\cE}(\cE)$. The identities written in the statement are satisfied by the character of a 2-dimensional representation, hence satisfied by $T_x$ for every classical point $x$. By the Zariski-density of classical points, $T$ satisfies the same identities. See \cite[Prop.~7.5.4]{bellaiche_chenevier_pseudobook} for more details. 
\end{proof}

\begin{remark}
	The identities in Proposition \ref{prop:pseudocharacter} mean (by definition) that $T$ is a 2-dimensional \emph{pseudocharacter}. When the coefficients are algebraically closed fields of odd characteristic, these pseudocharacters biject with isomorphism classes of 2-dimensional semisimple representations. See \cite{chenevier_det} for much more on this topic. 
\end{remark}

\subsection{Other topics}
Here are some brief comments on applications of families of overconvergent modular forms and the eigencurve:

\subsubsection{The infinite fern} If we have a Hecke eigenform $f$ of level $\Gamma_1(N)$, we get \emph{two} $U_p$-eigenforms at level $\Gamma_1(N)\cap\Gamma_0(p)$ whose Hecke eigenvalues away from $p$ match those of $f$. Gouv\^{e}a and Mazur observed that this makes the image of the eigencurve $\cE$ in a deformation space of Galois representations look like an `infinite fern' with fronds branching at points corresponding to Hecke eigenforms of level $\Gamma_1(N)$. They used this to prove density of points coming from modular forms in Galois deformation spaces, in some cases \cite{Maz97,GM-density}. This important result in the $p$-adic Langlands programme has since been generalized by B\"{o}ckle \cite{BigR=TBoeckle}, Chenevier \cite{Che11}, Breuil--Hellmann--Schraen \cite{Bre17}, Hellmann--Margerin--Schraen \cite{HMS-density} and Hernandez--Schraen \cite{hernandez-schraen}.

\subsubsection{$p$-adic $L$-functions}
The eigencurve provides a natural domain for families of $p$-adic $L$-functions for modular forms. This theory was initiated by Stevens, and you can read about it in Bella\"{i}che's Eigenbook \cite{eigenbook} (which is a great introduction to eigenvarieties).

\subsubsection{The eigencurve near the boundary of weight space}	
Coleman's remarkable `spectral halo' conjecture predicted that, after deleting the pre-image of closed subdiscs in the open unit discs making up $\cW$, the remainder of the eigencurve should break up as a disjoint union of finite flat covers of boundary annuli in $\cW$. This was proved in case $N=1, p=2$ by Buzzard and Kilford \cite{Buz05}. Liu, Wan and Xiao essentially proved Coleman's conjecture, working with the eigencurve for a definite quaternion algebra \cite{Liu17}\footnote{A proof for the full Coleman--Mazur eigencurve appeared very recently \cite{Diao-Yao}.}. Ye has proved some interesting results in higher dimension, where the precise picture is still quite unclear \cite{Ye-slopes}. The simple geometry of the eigenvariety in these boundary regions can be useful: Thorne and I use the Buzzard--Kilford result in our proof of symmetric power functoriality for modular forms \cite{NewTho}.

\section{An overview of constructions of eigenvarieties}\label{sec:overview}
In the previous section we have described how the eigenvariety machine explained in \cite{jl-heidelberg} can be applied to construct the eigencurve. We can ask how this generalizes to other contexts. The goal is to define adic spaces which $p$-adically interpolate the systems of Hecke eigenvalues attached to (a certain class of) automorphic representations of $G(\AA_{\Q})$, where $G$ is a fixed reductive group defined over $\Q$. 

In the next section, we give some details of the case where $G$ is given by the units in a definite quaternion algebra $D/\Q$. It is also instructive to consider the case $G = \mathrm{Res}_{F/\Q}\GL_1$ for a number field $F$, where the $\Spa(R,R^+)$-points of the eigenvariety with tame level a compact open subgroup $U^p \subset \AA_{F,f}^p$ correspond to continuous homomorphisms
\[F^\times \backslash \AA_F^\times /U^p \to R^\times.\] 
Both these cases were worked out in \cite{Buz04}.

In general, different methods to construct eigenvarieties are available depending on the automorphic context. The following list collects together some examples in the literature, grouped in terms of methods. We have omitted a discussion of the significant literature (much of it due to Hida) about the \emph{ordinary} or slope $0$ case.

\subsection{Coherent cohomology of Shimura varieties}
As we have seen, the Coleman--Mazur eigencurve was constructed using overconvergent modular forms, which are spaces of sections of line bundles on certain adic subspaces of modular curves. 

In contexts where we can associate a Shimura variety to the reductive group $G$, we can consider the coherent cohomology of automorphic vector bundles on this variety and try to generalize Coleman's theory.  The fact that Coleman relied on the $p$-adic family of Eisenstein series to define families of overconvergent modular forms was initially an obstruction to this, which was only overcome when more conceptual definitions of these families were discovered. Here are some of the works which construct eigenvarieties using coherent cohomology of Shimura varieties:

\begin{itemize}
	\item \cite{Col98, Buz07} (modular forms)
	\item \cite{kisin-lai,ais-hilbert, aip-hilbert} (Hilbert modular forms)
	\item \cite{aip-siegel} (Siegel modular forms)
	\item \cite{brasca, hernandez} (PEL Shimura varieties)
	\item \cite{aip-halo} (extends the eigencurve to a mixed characteristic analytic adic space)
	\item \cite{bp-higher} (works with coherent cohomology groups of all degree, not just $H^0$)	
\end{itemize}

We must add that when developing this theory, one also wants to generalize Coleman's classicality theorem. This has its own extensive literature.

\subsection{Singular cohomology of locally symmetric spaces} 
The Eichler--Shimura isomorphism tells us that we can find systems of Hecke eigenvalues for modular forms of weight $k \ge 2$ in the singular cohomology of modular curves, with coefficients in a local system depending on $k$. Stevens initiated the theory of \emph{overconvergent modular symbols} or \emph{overconvergent cohomology}, defining Banach spaces with a compact action of a $U_p$ operator by taking cohomology of modular curves with coefficients in a local system of Banach spaces. These can be used to construct eigenvarieties.

A more representation-theoretic approach, still based on singular cohomology, was developed by Emerton. Here, one key insight is that Hecke eigenforms which are `finite slope' at $p$ (i.e.~have a non-zero $U_p$-eigenvalue) have associated automorphic representations whose local factor $\pi_p$ admits a non-zero Jacquet module. Since we only considered modular forms of level $\Gamma_0(p)$ at $p$ in these notes, this condition was automatic (the local factors $\pi_p$ are either unramified principal series or unramified twists of Steinberg). Emerton's completed cohomology and locally analytic Jacquet functor gives another approach to constructing eigenvarieties.

In cases where our reductive group $G$ has $G(\RR)$ compact (or compact modulo centre), this theory simplifies since the relevant locally symmetric spaces are zero-dimensional. Examples include when $G$ is a definite unitary group or units in a definite quaternion algebra.

Here are some references which work with singular cohomology of locally symmetric spaces to construct eigenvarieties:

\begin{itemize}
	\item \cite{ste,Bel12,eigenbook} (overconvergent modular symbols and the eigencurve)
	\item \cite{Buz04,Buz07,chenevier-GLn,loeffler-ocaaf} (overconvergent cohomology in the zero-dimensional case)
	\item \cite{as,Han17,urban-eigen,xiang-full,JN-extended, BS-Williams} (overconvergent cohomology in general)
	\item \cite{MR2207783,hill-loeffler,breuil-ding} (Emerton's construction and generalizations)
	\item \cite{tarrach} (a different approach, can be compared with both overconvergent cohomology and Emerton's construction)
\end{itemize}

\subsection{Eigenvarieties and Galois deformation spaces}
In Proposition \ref{prop:pseudocharacter}, we observed a connection between eigenvarieties and $p$-adic families of Galois representations (or Galois pseudocharacters). This was already present in Coleman and Mazur's paper \cite{Col98}, and was used by them to give an alternative Galois theoretic construction of the eigencurve (although many of the good properties of the eigencurve are proved using its automorphic construction). 

The close relationship between eigenvarieties and Galois deformation spaces has been refined and generalized. In particular work of Kisin \cite{kisin-ocfmc} led to identifying the condition on the Galois side (\emph{trianguline} \cite{colmez-tri})  which, at least conjecturally, gives a precise description of the Galois representations coming from points of eigenvarieties. See for example \cite[\S6]{Han17}. Some more references are \cite{bellaiche_chenevier_pseudobook,hellmann2012families,Ked14,liu-triangulation,breuil-ding,Bre17}.

\subsection{Comparing different constructions of eigenvarieties}
Often in one automorphic setting there is more than one method available to construct an eigenvariety. For example, we can construct an eigencurve using overconvergent modular forms, but also using overconvergent modular symbols (see \cite{eigenbook}, for example). We would like to know that the resulting eigenvariety is independent of the way it is constructed. This can often be verified, using the density of classical points and the fact that we can identify the subsets of classical points the different eigenvarieties when they arise from the same collection of automorphic representations. Examples of precise statements can be found in \cite[\S 7.2]{eigenbook}, \cite[Proposition 7.2.8]{bellaiche_chenevier_pseudobook}, \cite[Theorem 5.1.2]{Han17}, \cite[Theorem 3.2.1]{JN-interp}, \cite[Theorem 5.18]{jl-heidelberg}. 

These results allow us to compare different kinds of eigenvarieties, but they do not produce maps between different kinds of spaces of overconvergent automorphic forms. To do the latter is more delicate --- for example, see \cite{AIS-ocES,CHJ,JERC} for results comparing overconvergent modular forms and overconvergent modular symbols. 

\section{Eigencurves for definite quaternion algebras} \label{sec:defquat}
\subsection{Automorphic forms for definite quaternion algebras}
In this section, we will explain the theory of overconvergent automorphic forms in a simpler context, where the definition of families and the proof of `small slope implies classical' will be relatively straightforward. We essentially follow Buzzard's article \cite{Buz04}, which was inspired by Stevens's overconvergent modular symbols. Our presentation has also been influenced by Chenevier's definition of overconvergent automorphic forms for definite unitary groups \cite{chenevier-GLn}.

We will work with a definite quaternion algebra $D/\Q$, so $D\otimes_{\Q}\R$ is isomorphic to Hamilton's quaternions $\mathbb{H}$. The discriminant $\delta_D$ of $D$ will be a product of an odd number of distinct primes, determined by: $D\otimes_{\Q}\Q_l \cong M_2(\Ql)$ for $l \nmid \delta_D$. When $l|\delta_D$, $D\otimes_{\Q}\Q_l$ is a division algebra of dimension $4$ over $\Q_l$. We fix once and for all a maximal order $\cO_D$ in $D$, and for each prime $l \nmid \delta_D$ an isomorphism $\cO_D\otimes_{\ZZ}\Zl \cong M_2(\Zl)$.

We can define an algebraic group $G/\Q$ with points in a $\Q$-algebra $R$ given by $G(\Q) = (D\otimes_{\Q}R)^\times$. This is an inner form of $\GL_2$, in particular the extension of scalars $G_{\Qbar} \cong \GL_2$. The theory of automorphic forms for $G$ is very simple, because $G(\R) \cong \mathbb{H}^\times$ is compact modulo centre, and hence its irreducible unitary representations are finite-dimensional. 

We fix a prime $p \nmid \delta_D$ and an isomorphism $D\otimes_{\Q}\Q_p \cong M_2(\Qp)$, which we use to identify $G(\Qp)$ and $\GL_2(\Qp)$. An important role will be played by the Iwahori subgroup $\Iw_p \subset \GL_2(\Zp)$ of matrices which are upper triangular mod $p$. This is the local counterpart to the congruence subgroup $\Gamma_0(p)$.

Our spaces of classical and $p$-adic automorphic forms will all be defined using the following general construction:

\begin{defn}
	Let $K = \prod_{l}K_l$ be a compact open subgroup of $G(\A_f)$ with $K_p \subset \Iw_p$. Write $K^p$ for the product $\prod_{l \ne p}K_l \subset G(\A_f^{(p)})$. 
	
	Let $A$ be an abelian group with a linear left action of $\Iw_p$. We define
	
	\[\cL(K,A) := \{f: G(\Q)\backslash G(\A_f) \to A \mid f(gk) = k_p^{-1}f(g) \forall k \in K\}.\]
\end{defn}

We can view $\cL(K,A)$ as the $K$-invariants in the space of functions \[f: G(\Q)\backslash G(\A_f) \to A,\] equipped with the left action of $G(\A_f^{(p)}) \times \Iw_p$ as follows: \[(\gamma\cdot f)(g) = \gamma_pf(g\gamma)\]

As usual, if $g$ is an adelic thing, we write $g_p$ for its component at $p$.

\begin{prop}\label{prop:cosets}
	The set $G(\Q)\backslash G(\A_f)/K$ is finite. Choose a set of double coset representatives $(g_i)_{i=1}^d \subset G(\A_f)$ for this set. The map	
	\begin{align*} \cL(K,A) &\to A^{\oplus d} \\ f &\mapsto f(g_i)\end{align*} defines an isomorphism $\cL(K,A) \cong \bigoplus_{i=1}^{d} A^{\Gamma_i}$, where $\Gamma_i = K\cap\left(g_i^{-1}G(\Q)g_i\right)$.
	
	The groups $\Gamma_i$ are finite.
\end{prop}
\begin{proof}
	Finiteness of $G(\Q)\backslash G(\A_f)/K$ is a generalisation of finiteness of class groups; it was proved for general algebraic groups $G/\Q$ by Borel (using work with Harish-Chandra) \cite{borel-finite}. Checking that the map is an isomorphism is an exercise. The groups $\Gamma_i$ are finite because they are intersections of a compact and a discrete subgroup of $G(\A)$.
\end{proof}
If we try to use this setup for $G = \GL_2$, or for $G$ coming from an \emph{indefinite} quaternion algebra, we will not get anything interesting.  That's because $\SL_2$, or the reduced norm $1$ subgroup of $G$ in the indefinite case, satisfies strong approximation away from $\infty$: the group $\SL_2(\Q)$ is dense in $\SL_2(\A^\infty)$. So a function in $\cL(K,A)$ will be forced to factor through the determinant or reduced norm map to $\GL_1(\AA_f)$.

We can equip the spaces $\cL(K,A)$ with double coset operators: if $g\in G(\A_f^{(p)})$, write the double coset $KgK$ as a disjoint union of finitely many left cosets \[Kg K = \coprod_{i=1}^{d_g} g_iK\] and define \[ [KgK]f = \sum g_i\cdot f.\]

\begin{exercise}
	Check that the formula for the double coset operator $[KgK]$ does indeed define an endomorphism of $\cL(K,A)$. 
\end{exercise}

When $l\nmid \delta_Dp$, we let $t_l = \left(\begin{smallmatrix}
	l & 0 \\ 0 & 1
\end{smallmatrix}\right)$ and $s_l = \left(\begin{smallmatrix}
	l & 0 \\ 0 & l
\end{smallmatrix}\right)$, which we can view as elements of $G(\Ql)$. Supposing $K_l = \GL_2(\Zl)$, we define Hecke operators $T_l = [Kt_lK]$, $S_l = [Ks_lK]$. 
When the $\Iw_p$-action on $A$ extends to a monoid containing $u_p:=\left(\begin{smallmatrix}
	p & 0 \\ 0 & 1
\end{smallmatrix}\right)$, we can also define $U_p = [Ku_pK]$. 

Here is an important example of a choice of $A$: 

\begin{defn}
	If $a \ge b$ are integers, we consider the representation $V_{a,b} = \Sym^{a-b}\Qp^{{2}}\otimes\det^b$ of $\GL_2(\Qp)$, which acts by the standard representation on $\Qp^2$.
\end{defn}

The representation $V_{a,b}$ is the algebraic representation of highest weight $(a,b)$. We can now compare with spaces of modular forms. Fix a level $N$, coprime to $\delta_Dp$. We set $K^p = K_1(N)$ and $K_p = \Iw_p$, where $K_1(N)$ is the compact open subgroup of $\left(\cO_D\otimes_\Z \widehat{\Z}\right)^\times$ defined by $\left\{g \bmod N \in \left(\begin{smallmatrix}
	* & * \\ 0 & 1
\end{smallmatrix}\right) \subset \GL_2((\Z/N\Z)^\times)\right\}.$ 

\begin{theorem}[Jacquet--Langlands, see \cite{DT}]\label{thm:JL}
	Fix an integer $k \ge 2$ and an embedding of (abstract) fields $\iota: \Qp \hookrightarrow \C$. There is a map
	\[\JL:\cL(K,V_{k-2,0})\otimes_{\iota}\C \to S_k(\Gamma_1(N)\cap\Gamma_0(\delta_Dp))\] to the space of cuspidal modular forms of weight $k$, with the following properties:
	\begin{itemize}
		\item $\JL$ is equivariant for the Hecke operators $T_l, S_l, U_p$.
		\item $\JL$ is injective when $k > 2$. When $k = 2$, the kernel is given by functions which factor through the reduced norm $G(\A_f) \to \A_f^\times$.
		\item The image of $\JL$ is the modular forms which are `$\delta_D$-new'.
	\end{itemize}  
\end{theorem}

The spaces $\cL(K,V_{k-2,0})$ therefore give us a very concrete way to explore the Hecke eigenvalues of modular forms. For those familiar with automorphic representations, there is a Hecke-equivariant isomorphism
\[\cL(K,V_{k-2,0})\otimes_{\iota}\C \cong  \oplus_{\pi} (\pi^\infty)^K,\] where the sum is over automorphic representations $\pi$ of $G(\A)$ with $\pi_\infty \cong V_{k-2,0}^\vee\otimes_{\iota}\C$. Here we are viewing the representation $V_{k-2,0}^\vee\otimes_{\iota}\C$ of $\GL_2(\C)$ as a representation of $G(\R)$ using the embedding $G(\R) \subset G(\C) \cong \GL_2(\C)$. This is explained carefully in \cite[\S2]{DT}.

\subsection{Overconvergent automorphic forms for definite quaternion algebras}\label{ssec:ocquat}
If we want to define families of $p$-adic automorphic forms, we need to replace $V_{k-2,0}$ with a family of Banach modules over the weight space $\cW = \Hom(\Zp^\times,\mathbb{G}_m)$, whose specializations at integral weights $[k]$ contain $V_{k-2,0}$. 

We can intepret $V_{k-2,0}$ as an algebraic induced representation $\Ind_{B}^{\GL_2}\lambda_{0,k-2}$, where $\lambda_{a,b}(\mathrm{diag}(t_1,t_2)) = t_1^at_2^b$ denotes a weight for the diagonal torus $T$ in $\GL_2$. The weight $\lambda_{0,k-2}$ is anti-dominant with respect to the upper triangular Borel subgroup $B$, and is the lowest weight of the representation $V_{k-2,0}$.

\begin{defn}
	\begin{align*}\Ind_{B}^{\GL_2}\lambda_{0,k-2} &= \{f \in \cO(\GL_{2,\Qp})\mid f(tng) = t_2^{k-2}f(g) \\ &\quad\quad\text{ for all }t=\mathrm{diag}(t_1,t_2)\in T, n \in N.\}\end{align*}
\end{defn}
The action of $\GL_2(\Qp)$ is by right translation. When we write $\cO(\GL_{2,\Qp})$ we mean the algebraic functions on this algebraic group. In particular, when we restrict to the strictly lower-triangular subgroup $\overline{N}$ we get polynomials in one variable.
\begin{exercise}\label{ex:induction}
	Show that restricting functions to $\overline{N}(\Qp) = \left\{\left(\begin{smallmatrix}
		1 & 0 \\ x & 1
	\end{smallmatrix}\right) : x \in \Qp\right\}$ gives an isomorphism between $\Ind_{B}^{\GL_2}\lambda_{0,k-2}$ and polynomials in $\Qp[x]$ of degree $\le k-2$, with action of $\GL_2(\Qp)$ given by:
	
	\[\left(\left(\begin{smallmatrix}
		a & b \\ c & d
	\end{smallmatrix}\right)f\right)(x) = (bx+d)^{k-2}f((ax+c)/(bx+d)).\]
	
\end{exercise}

We want to enlarge our representation space $V_{k-2,0}$ so that it will vary nicely in a family of weights. Motivated by Exercise \ref{ex:induction}, we will do this using functions defined on lower unipotent matrices.

\begin{defn}
	Write $\overline{N}_1$ for the subgroup \[p\Zp\cong\begin{pmatrix}
		1 & 0\\ p\Z_p & 1
	\end{pmatrix} \subset \overline{N}(\Qp).\]
	For $v \in \Z_{\ge 0}$, write $\cC^{v}_{\mathrm{an}}(\overline{N}_1,\Qp)$ for the continuous functions $f:\overline{N}_1 \to \Qp$ which are \emph{locally analytic with radius of analyticity $p^{-v}$}. 
	
	This means that the restriction of $f$ to $p(a+p^v\Zp)$ is given by a convergent power series $f(p(a+p^vz)) = f_a(z) \in \Qp\langle z \rangle$ for all $a \in \Zp$. When $v = 0$, our functions are defined by a single power series.
	
	$\cC^{v}_{\mathrm{an}}(\overline{N}_1,\Qp)$ is a $\Qp$-Banach space with the supremum norm. Note that this supremum norm can also be described as $\sup_{a \in \Z/p^v\Z}|f_a|$ for the Gauss norm on $\Qp\langle z \rangle$. 
\end{defn}

We can think of $\cC^{v}_{\mathrm{an}}(\overline{N}_1,\Qp)$ as the ring of functions on the disjoint union of closed discs $\mathbb{B}_v:=\coprod_{a \in \Z/p^v\Z}\{|x-pa| \le |p|^{1+v}\}$. 

We then define locally analytic induced representations of $\Iw_p$:

\begin{defn}
	For integers $k_1, k_2$, define \begin{align*}\cA^v_{k_1,k_2} &= \Ind_{B(\Zp)}^{\Iw_p}(\lambda_{k_2,k_1})^{v-an} := \{f: \Iw_p \to \Qp \mid f(tng) = \lambda_{k_2,k_1}(t)f(g) \\&\text{for all }t \in T(\Zp), n \in N(\Zp), \text{and }f|_{\overline{N}_1} \in \cC^{v}_{\mathrm{an}}(\overline{N}_1,\Qp)\}.\end{align*}
\end{defn} 
As usual, $\Iw_p$ acts on these functions by right multiplication. When $k_1 \ge k_2$, restriction of functions to $\Iw_p$ defines an injective map 
\begin{align*}
	V_{k_1,k_2} &\hookrightarrow \cA^v_{k_1,k_2}\\
	\Ind_{B}^{\GL_2}\lambda_{k_2,k_1} & \hookrightarrow \Ind_{B(\Zp)}^{\Iw_p}(\lambda_{k_2,k_1})^{v-an} 
\end{align*}

\begin{exercise}
	Using the Iwahori decomposition $\Iw_p = B(\Zp)\overline{N}_1$, show that: \begin{enumerate}
		\item Restriction to $\overline{N}_1$ defines an isomorphism $\cA^v_{k_1,k_2} \cong \cC^{v}_{\mathrm{an}}(\overline{N}_1,\Qp)$. 
		\item Applying the isomorphism in the previous part, the action of $\Iw_p$ on $\cC^{v}_{\mathrm{an}}(\overline{N}_1,\Qp)$ is given by
		\[\left(\left(\begin{smallmatrix}
			a & b \\ c & d
		\end{smallmatrix}\right)f\right)(x) = (ad-bc)^{k_2}(bx+d)^{k_1-k_2}f((ax+c)/(bx+d)).\] Check that this makes sense, using the fact that $x \mapsto (ax+c)/(bx+d)$ defines a map of adic spaces  $\mathbb{B}_v \to \mathbb{B}_v$. 
	\end{enumerate}
\end{exercise}

\begin{defn}
	The space of $v$-overconvergent automorphic forms of weight $(k_1,k_2)$ and level $K$ is defined to be:
	$S^{\dagger,v}_{k_1,k_2}(K) := \cL(K,\cA^v_{k_1,k_2})$. 
	
	Using the isomorphism $\cL(K,\cA^v_{k_1,k_2}) \cong \bigoplus_{i=1}^{d} (\cA^v_{k_1,k_2})^{\Gamma_i}$, we can make $S^{\dagger,v}_{k_1,k_2}(K)$ into a $\Qp$-Banach space.
\end{defn}
The space $S^{\dagger,v}_{k_1,k_2}(K)$ is equipped with Hecke operators $T_l, S_l$ for $l \nmid \delta_DpN$, given by double coset operators. 
We need to extend the action of $\Iw_p$ on $\cA^v_{k_1,k_2}$ in order to define a Hecke operator $U_p$. 

When $k_1 \ge k_2$, let's think about the embedding $V_{k_1,k_2} \hookrightarrow \cA^v_{k_1,k_2}$ and the action of an element $\gamma=\left(\begin{smallmatrix}
	p^i & 0 \\ 0 & p^j
\end{smallmatrix}\right)$. For the $\gamma$-action on $V_{k_1,k_2}$ to have a hope of extending nicely to $\cA^v_{k_1,k_2}$, we need $\gamma^{-1}\overline{N}_1\gamma \subset \overline{N}_1$. This ensures that $(\gamma f)|_{\overline{N}_1}$ is determined by $f|_{\overline{N}_1}$. This condition is equivalent to $i \ge j$, so let's assume that.

Now, thinking of $V_{k_1,k_2}$ as polynomial functions on $\overline{N}_1$ of degree $\le k_1-k_2$, how does $\gamma$ act? We get $(\gamma f)(x) = p^{ik_2 + jk_1}f(p^{i-j}x)$.

To get an action which interpolates nicely as we vary $(k_1,k_2)$ $p$-adically, we are going to need to remove that power of $p$ at the front! So we define a rescaled action of the elements $u_p^{i,j}= \left(\begin{smallmatrix}
	p^i & 0 \\ 0 & p^j
\end{smallmatrix}\right)$: $u_p^{i,j}f = f(p^{i-j}x)$. This also makes sense when $k_1 < k_2$. Note that when $(k_1,k_2)= (k-2,0)$, the action of our favourite element $u_p = u_p^{1,0}$ didn't get rescaled at all. Collecting everything together, we have defined an action on $\cA_{k_1,k_2}^v$ of the monoid $\Delta^+$ generated by $\Iw_p$ and the elements $u_p^{i,j}$ with $i \ge j$. The action of $u_p^{i,j}$ is induced by (right) conjugation on $\overline{N}_1$.

With all this done, we can define an action of $U_p$ with the double coset operator $U_p = [Ku_pK]$ on $S^{\dagger,v}_{k_1,k_2}(K)$. 

\begin{lem}\label{lem:Upcptquat}
	$U_p$ is a compact operator, and is norm-decreasing.
\end{lem}
\begin{proof}
	The map $x \mapsto px$ defines a map $\mathbb{B}_v \to \mathbb{B}_{v+1}$, so the action of $u_p$ on $\cA^v_{k_1,k_2}$ factors as $\cA^v_{k_1,k_2} \to \cA^{v+1}_{k_1,k_2} \xrightarrow{u_p} \cA^v_{k_1,k_2}$ and $\cA^v_{k_1,k_2} \xrightarrow{u_p} \cA^{v-1}_{k_1,k_2} \to \cA^v_{k_1,k_2}$ where the unlabelled maps are embedding functions with one radius of analyticity into the space of functions with a smaller radius of analyticity.
	
	This shows that $U_p$ factors through the inclusion $S^{\dagger,v}_{k_1,k_2}(K) \to S^{\dagger,v+1}_{k_1,k_2}(K)$ and is therefore compact, since the inclusion $\cA^v_{k_1,k_2} \to \cA^{v+1}_{k_1,k_2}$ looks like one of our basic examples of compact maps (the map on affinoid algebras corresponding to inclusion of a closed disc inside a closed disc with bigger radius).
\end{proof}
Collecting together all our Hecke operators, we have defined an action of $\TT^{(\delta_DN)}$ on $S^{\dagger,v}_{k_1,k_2}(K)$.

A very nice feature of this theory is that it is easy to identify classical automorphic forms inside the space of overconvergent automorphic forms:

\begin{defn}
	For $t \in \Z_{\ge 0}$, we define a differential operator $D_t$ on the space $\cC^{v}_{\mathrm{an}}(\overline{N}_1,\Qp)$ by $(D_t f)(x) = \frac{d^tf}{dx^t}(x)$.
\end{defn}
More formally, we differentiate $f$ by restricting to the residue discs where it is analytic, differentiating there, and gluing the derivatives back together to give us a locally analytic function on all of $\overline{N}_1$.

\begin{theorem}\label{thm:bggquat}
	Suppose $k_1 \ge k_2$. There is a $\Iw_p$-equivariant exact sequence, 
	\[0 \to V_{k_1,k_2} \to \cA^0_{k_1,k_2} \xrightarrow{D_{k_1-k_2+1}} \cA^0_{k_2-1,k_1+1}\]
	
	and an induced exact sequence 
	\[0 \to \cL(K,V_{k_1,k_2}) \to \cL(K,\cA^0_{k_1,k_2}) \xrightarrow{D_{k_1-k_2+1}} \cL(K,\cA^0_{k_2-1,k_1+1}).\]
	
	We have $D_{k_1-k_2+1}\circ U_p = p^{k_1-k_2+1}U_p\circ D_{k_1-k_2+1}$.
\end{theorem}
\begin{proof}
	The kernel of the differential operator is given by polynomial functions of degree $\le k_1-k_2$, which is exactly the subspace $V_{k_1,k_2}$ in $\cA^0_{k_1,k_2}$\footnote{We need $v=0$, otherwise we just get locally polynomial functions, i.e.~a polynomial function on each $p^v\Zp$ coset. Thanks to Peter Gr\"{a}f for pointing this out in the lectures.}. The functor $\cL(K,-)$ is exact (by Proposition \ref{prop:cosets}). Checking the $\Iw_p$ and $U_p$-equivariance properties is an exercise. 
\end{proof}

\begin{remark}The weight $(k_2-1,k_1+1)$ is what you get by apply the `dot' action of the non-trivial Weyl group element to $(k_1,k_2)$, where $w\cdot\lambda = w(\lambda+\rho)-\rho$, with $\rho$ half the sum of the positive roots.\end{remark}

\begin{cor}[`Small slope implies classical']
	
	If $h < k - 1$, the map \[\cL(K,V_{k-2,0})^{\le h} \to \cL(K,\cA^v_{k-2,0})^{\le h}\] is an isomorphism.
\end{cor}
\begin{proof}
	First note that $\cL(K,\cA^v_{k-2,0})^{\le h} = \cL(K,\cA^0_{k-2,0})^{\le h}$, because $U_p$ increases the radius of analyticity and therefore the bounded slope spaces are independent of $v$. We need to show that $\cL(K,\cA^0_{k-2,0})^{\le h}$ is contained in the kernel of $D_{k-1}$. Since $D_{k-1}\circ U_p = p^{k-1}U_p\circ D_{k-1}$, the operator $D_{k-1}$ gives a map:
	\[D_{k-1}: \cL(K,\cA^0_{k-2,0})^{\le h} \to \cL(K,\cA^0_{-1,k-1})^{\le h-(k-1)}.\] But $U_p$ is norm-decreasing on $\cL(K,\cA^0_{-1,k-1})^{\le h-(k-1)}$, so if $h < k-1$ this subspace is zero ($U_p$ cannot have eigenvalues with negative slope). 
\end{proof}

\subsection{Overconvergent automorphic forms in families}
We can now define overconvergent automorphic forms with an affinoid weight
$\kappa: \Zpx \to R^\times$, corresponding to a map from an affinoid adic space $\Spa(R,R^+)=U \to \cW$. 

We are going to define \begin{align*}\cA^v_{\kappa} = \Ind_{B(\Zp)}^{\Iw_p}(1\otimes\kappa)^{v-an} &:= \{f: \Iw_p \to R \mid f(tng) = \kappa(t_2)f(g) \\&\text{for all }t \in T(\Zp), n \in N(\Zp), \text{and }f|_{\overline{N}_1} \in \cC^{v}_{\mathrm{an}}(\overline{N}_1,R)\}.\end{align*}

So we need to define $\cC^{v}_{\mathrm{an}}(\overline{N}_1,R)$: this is simply $\cC^{v}_{\mathrm{an}}(\overline{N}_1,\Qp)\widehat{\otimes}_{\Qp}R$, which we can identify with `$v$-analytic' $R$-valued functions on $\overline{N}_1$, and also with the affinoid algebra of functions on $\mathbb{B}_v\times_{\Spa(\Qp,\Zp)}U$. Fixing a norm $|.|_R$ on the Tate algebra $R$ inducing its topology, we make $\cC^{v}_{\mathrm{an}}(\overline{N}_1,R)$ into a potentially ONable Banach $R$-module. We can define the norm $|f| = \sup_{x \in \overline{N}_1}|f(x)|_R$.

We need to check that $\cA^v_{\kappa}$ is stable under the action of $\Iw_p$. This means that the formula 
\[\left(\left(\begin{smallmatrix}
	a & b \\ c & d
\end{smallmatrix}\right)f\right)(x) = \kappa(bx+d)f((ax+c)/(bx+d))\] must define a function in $\cC^{v}_{\mathrm{an}}(\overline{N}_1,R)$. So we need $x \mapsto \kappa(bx + d)$ to be locally analytic with radius of analyticity $p^{-v}$ for all $b \in \Zp, d \in \Zpx$. For any fixed $\kappa$ this will be true for $v$ sufficiently large (see \cite[Theorem 6.3.4]{eigenbook} again), in which case we say that $\kappa$ is $v$-analytic.

\begin{defn}
	Suppose $\kappa : \Zpx \to R^\times$ is a $v$-analytic affinoid weight. Then the space of $v$-overconvergent automorphic forms of weight $\kappa$ and level $K$ is defined to be:
	$S^{\dagger,v}_{\kappa}(K) := \cL(K,\cA^v_{\kappa})$. 
\end{defn}
Using the isomorphism $\cL(K,\cA^v_{\kappa}) \cong \oplus_{i=1}^d (\cA^v_{\kappa})^{\Gamma_i}$ (which depends on a choice of double coset representatives $(g_i)_{i=1}^d$), we see that $S^{\dagger,v}_{\kappa}(K)$ is a Banach $R$-module with property (Pr). This is because each $\Gamma_i$ is finite, so the $\Gamma_i$-invariants of an $R$-module are a direct summand. We can use the norm $|f| = \sup_{1\le i \le d}|f(g_i)|$.

Just as in the integral weight case, the spaces $S^{\dagger,v}_{\kappa}(K)$ come with actions of the Hecke algebra $\TT^{(\delta_DN)}$, and the action of $U_p$ is compact and norm-decreasing. 

\begin{exercise}
	Use Proposition \ref{prop:cosets} to show that $S^{\dagger,v}_{\kappa}(K)$ is compatible with base change of weights. In other words, whenever there is an affinoid weight $\kappa_U: \Zpx \to \cO(U)^\times$ and a map of affinoids $V \to U$ inducing an affinoid weight $\kappa_V: \Zpx \to \cO(V)^\times$, we have
	$S^{\dagger,v}_{\kappa_U}(K)\widehat{\otimes}_{\cO(U)}\cO(V) \cong S^{\dagger,v}_{\kappa_V}(K)$.
\end{exercise}

\subsection{Eigencurves for definite quaternion algebras}
We now have what we need to construct an eigencurve $\cE^D = \cE^D_{\delta_DN} \xrightarrow{w} \cW$ from the modules $S^{\dagger,v}_{\kappa_U}(K)$ with their Hecke algebra action. It shares all the nice properties of the Coleman--Mazur eigencurve which we discussed in Section \ref{sec: geom props}. In particular, it is reduced and comes with a Zariski dense and self-accumulating subset $Z^D \subset \Max(\cE^D)$ of classical points corresponding to systems of Hecke eigenvalues appearing in $\cL(K,V_{k-2,0})$ for $k \ge 2$.

\subsubsection{Comparing eigenvarieties}\label{comparison}
The comparison between quaternionic eigencurves and the Coleman--Mazur eigencurve was first established by Chenevier \cite{chenevier-JL}.
\begin{thm}
	There is a closed immersion \[\JL: \cE^D_{\delta_DN} \hookrightarrow \cE_{\delta_DN},\] uniquely determined by the property that it maps a classical point in $\cE^D_{\delta_DN}$ (excluding the points coming from functions on $G(\A_f)$ which factor through the reduced norm) to its corresponding point in $\cE_{\delta_DN}$ under the Jacquet--Langlands correspondence (Theorem \ref{thm:JL}).
\end{thm}
\begin{proof}
	Our two eigencurves come with finite maps to spectral curves $\cZ, \cZ^D$ which have closed immersions $\iota, \iota^D$ to $\AA^1_{\cW}$, and coherent sheaves $\cN, \cN^D$ with actions of $\TT^{(\delta_DN)}$ built using the spaces of overconvergent modular forms and Riesz theory. In other words, these eigenvarieties are constructed using the eigenvariety data $(\cZ,\cN,\TT^{(\delta_DN)})$ and $(\cZ^D,\cN^D,\TT^{(\delta_DN)})$. To compare them, we can alternatively view the eigenvarieties as being constructed from the data $(\AA^1_{\cW},\iota_*\cN,\TT^{(\delta_DN)})$ and $(\AA^1_{\cW},\iota^D_*\cN^D,\TT^{(\delta_DN)})$. 
	
	A simple interpolation theorem \cite[Theorem 5.18]{jl-heidelberg} (interpolation results like this go back to Chenevier and Hansen), then gives the closed immersion $\JL$. It is unique because $\cE^D_{\delta_DN}$ is reduced and we have specified the map on a Zariski dense set of points.
\end{proof}

It follows from the description of the Jacquet--Langlands correspondence that the image of $\JL$ is a union of irreducible components whose classical points correspond to $\delta_D$-new Hecke eigenforms of level $\Gamma_1(N)\cap\Gamma_0(\delta_DN)$. 

\begin{exercise}
	Can you identify the image under $\JL$ of the points coming from functions on $G(\A_f)$ which factor through the reduced norm?
\end{exercise}

\subsubsection{Etaleness of the weight map at small slope classical points}
As a final topic in this section, let's study the weight map $w: \cE^D \to \cW$ at a classical point $x \in \cE^D$ with $w(x) = [k]$ and $U_p$-eigenvalue $\alpha$ of valuation $< k-1$. The Hecke eigenform corresponding to $x$ has a Nebentypus character $\epsilon: (\Z/N\Z)^\times \to k(x)^\times$, and we assume that $\alpha^2 \ne p^{k-1}\epsilon(p)$. We call these points \emph{regular} classical points\footnote{This is not standard terminology.}.

\begin{thm}
	Let $x \in \cE^D$ be a regular classical point. Assume, for simplicity\footnote{See \cite[Theorem 7.6.4]{eigenbook} for a similar statement without this assumption.}, that the Hecke eigenspace corresponding to $x$ in $\cL(K,V_{k-2,0})\otimes_{\Qp}k(x)$ is one-dimensional. Then the weight map is \'{e}tale at $x$.
\end{thm}
\begin{proof}
	Fix our classical point $x$, with $w(x) = [k]$ and $U_p$-slope $h < k-1$. We set $L = k(x)$ and extend scalars of all our adic spaces to $L$, so that we can identify $x$ with a geometric point of $\cE^D_L$. We will show that $w$ is an isomorphism over an open neighbourhood of $x$ in $\cE^D_L$. 
	
	We can find\footnote{See \cite[\S7.6.1]{eigenbook}.} a connected open affinoid neighbourhood $V$ of $[k]$ in $\cW_L$ such that $S^{\dagger,v}_{\kappa_V}(K)$ has a slope $\le h$ decomposition, and if we set \[\cO(U) = \im\left(\cO(V)\otimes_{\Qp}\TT^{(\delta_DN)} \to \End_{\cO(V)}(S^{\dagger,v}_{\kappa_V}(K)^{\le h})\right),\] then the connected component $U_x$ of $U$ containing $x$ satisfies $w^{-1}([k])\cap U_x = \{x\}$. 
	
	There is a direct summand $P$ of $\End_{\cO(V)}(S^{\dagger,v}_{\kappa_V}(K)^{\le h})$ such that \[\O(U_x) = \im\left(\cO(V)\otimes_{\Qp}\TT^{(\delta_DN)} \to \End_{\cO(V)}(P)\right).\]
	
	Since $x$ is the only geometric point of $w^{-1}([k])\cap U_x$, the fibre $P\otimes_{\cO(V),[k]}L$ is equal to the generalized eigenspace corresponding to $x$ in $S^{\dagger,v}_{k-2,0}(K)^{\le h}\otimes_{\Qp} L$. By classicality, this is the same as the generalized eigenspace in $\cL(K,V_{k-2,0})\otimes_{\Qp}L$. As we discussed in the proof of reducedness, this generalized eigenspace is actually an eigenspace (our Hecke operators act semisimply, under the regular assumption). We've assumed this eigenspace is one-dimensional, so  $P\otimes_{\cO(V),[k]}L$ is one-dimensional. Since $P$ is a projective $\cO(V)$-module, we deduce that $P$ is locally free of rank one. This implies that the natural map $\cO(V) \to \End_{\cO(V)}(P)$ is an isomorphism, and therefore the embedding $\cO(U_x) \to \End_{\cO(V)}(P)$ gives an $\cO(V)$-algebra isomorphism $\cO(U_x)\cong \cO(V)$. In other words $w$ induces an isomorphism $w: U_x \cong V$. 
\end{proof}

\section{Overconvergent cohomology}\label{sec:occoh}
In the final section, we'll describe how the above construction for definite quaternion algebras generalizes, following Ash, Stevens \cite{Ash-Stevens} and Hansen \cite{Han17} (a related construction is due to Urban \cite{urban-eigen}). For concreteness, we will fix attention on the group $\GL_n/\Q$, but this can be replaced with a more general reductive group. We fix the upper-triangular Borel subgroup $B \subset \GL_n$, diagonal maximal torus $T$ and Levi decomposition $B = TN$. The Iwahori subgroup $\Iw_p$ is the inverse image of $B(\Fp)$ in $\GL_n(\Zp)$. 

\subsection{Locally symmetric spaces and automorphic local systems.}
Locally symmetric spaces are generalisations of modular curves, and their singular cohomology can be described in terms of automorphic representations. For $K \subset \GL_n(\A_f)$ compact open, define

\[Y_K := \GL_n(\Q)\backslash \GL_n(\A) / \R_{>0}\SO_n(\R)K.\] 

When $K$ is sufficiently small, $\GL_n(\Q)$ acts freely on $\GL_n(\A) / (\R_{>0}\SO_n(\R)K)$. We've already mentioned Borel's result that $\GL_n(\Q)\backslash \GL_n(\A_f) / K$ is finite. If we choose double coset representatives $(g_i)_{i=1}^d$, then we get an isomorphism \begin{align*}\coprod_{i=1}^d \Gamma_i\backslash X^{\GL_n} &\cong Y_K  \\
	\Gamma_i x &\mapsto \GL_n(\Q)(x,g_i)K  \end{align*}
describing $Y_K$ as a disjoint union of arithmetic quotients of the symmetric space $X^{\GL_n} = GL_n(\R) / \R_{>0}\SO_n(\R)$, where $\Gamma_i = g_iKg_i^{-1}\cap\GL_n(\Q)$. 

For a $B$-dominant weight $\lambda$ of $\GL_n$, we have $\lambda(\mathrm{diag}(t_1,\ldots,t_n)) = \prod t_i^{\lambda_i}$ for a decreasing sequence of integers $\lambda_1 \ge \lambda_2 \ge \cdots \ge \lambda_n$. Write $V_{\lambda,\Q}$ for the highest weight $\lambda$ representation of $\GL_n(\Q)$. 

Then we have a local system of $\Q$-vector spaces, $\cV_{\lambda,\Q}$ on $Y_K$ defined by the quotient \[\GL_n(\Q)\backslash \left(V_{\lambda,\Q} \times \GL_n(\A) / \R_{>0}\SO_n(\R)K\right) \] with the diagonal action of $\GL_n(\Q)$.

We can alternatively define a $p$-adic local system $\cV_{\lambda,\Qp}$ using the representation $V_{\lambda,\Qp}$ of $\GL_n(\Qp)$ by taking the quotient \[\left(\left(\GL_n(\Q)\backslash \GL_n(\A) / \R_{>0}\SO_n(\R)K^p\right) \times V_{\lambda,\Q_p}\right)/K_p \] where $K_p$ acts by $(x,v)k = (xk,k^{-1}v)$. 

\begin{exercise}
	Show that the map \[(v,g) \mapsto (g,g_p^{-1}v)\] gives an isomorphism $\cV_{\lambda,\Q}\otimes_{\Q}\Qp \cong \cV_{\lambda,\Qp}$.
\end{exercise}

These cohomology groups can be equipped with Hecke operators using double coset operators, just like in the quaternion algebra case. We fix a set of primes $S$ containing $p$, such that $K_l = \GL_n(\Zl)$ for all $l \notin S$, and assume that $K_p = \Iw_p$.  We will use the abstract Hecke algebra over $\Q$, \[\TT^S_\Q := \left(\prod_{l \notin S} \cH(\GL_n(\Ql),\GL_n(\Zl))\otimes_{\Z}\Q\right)\otimes\Q[U_p]\] containing the spherical Hecke algebras at $l \notin S$ and the operator\footnote{It is more natural to add a bigger, commutative, Hecke algebra at $p$, but we'll fix just this one operator for simplicity.} \[U_p = [\Iw_p \mathrm{diag}(p^{n-1},p^{n-2},\ldots,p,1)\Iw_p].\] Later we will extend scalars to $\Qp$. 

It is a theorem of Franke \cite{franke} that the cohomology groups $H^*(Y_K,\cV_{\lambda,\C})$ can be computed in terms of automorphic representations of $\GL_n(\A)$. The description is quite complicated, but there is a simple description of the subspace of the cohomology groups coming from cuspidal automorphic representations:

\begin{thm}[Borel, Clozel \cite{clo90}, \ldots]\label{clozelcoh}
	There is a subspace \[H^*_{\mathrm{cusp}}(Y_K,\cV_{\lambda,\C}) \subset H^*(Y_K,\cV_{\lambda,\C})\] with a $\TT^S_\Q$-equivariant isomorphism 
	\[H^i_{\mathrm{cusp}}(Y_K,\cV_{\lambda,\C}) \cong \bigoplus_{\pi}\left((\pi^\infty)^K\right)^{\oplus m_{i,\lambda}(\pi_\infty)}.\]
	
	The sum is over cuspidal automorphic representations of $\GL_n(\A)$ and the multiplicities $m_{i,\lambda}(\pi_\infty)$ can only be non-zero if $\pi_\infty$ is \emph{cohomological of weight} $\lambda$. We must then have $\lambda = (\lambda_1 \ge \lambda_2 \ge \cdots \lambda_n)$ with $\lambda_i + \lambda_{n+1-i} = w$ independent of $i$\footnote{This is a consequence of Clozel's `purity lemma'.}. In other words, the dual representation $V^\vee_{\lambda}$ is a twist of $V_{\lambda}$.
	
	Moreover, when the $m_{i,\lambda}(\pi_\infty)$ are non-zero, they are given by the formula:
	
	\begin{align*}
		n = 2m \text{ even} :~ &m_{i,\lambda}(\pi_\infty) = \binom{m-1}{i-m^2}\\
		n = 2m+1 \text{ odd} :~ &m_{i,\lambda}(\pi_\infty) = \binom{m}{i-m(m+1)}
	\end{align*}
\end{thm}
Note that the range of degrees where cohomology can be non-zero is of the form $[q_0,q_0+l_0]$, with $l_0 = \lfloor\frac{n-1}{2}\rfloor$.

It has been shown by Harder and Raghuram \cite[\S5.1]{harder-raghuram} that the cuspidal subspace $H^i_{\mathrm{cusp}}(Y_K,\cV_{\lambda,\C})$ is the extension of scalars to $\C$ of a Hecke-stable $\Q$-subspace $H^i_{\mathrm{cusp}}(Y_K,\cV_{\lambda,\Q}) \subset H^i(Y_K,\cV_{\lambda,\Q})$. 

As in the quaternionic case, we now interpolate the spaces $H^*(Y_K,\cV_{\lambda,\Qp})$ by considering big representations $\cA_{\lambda}^v$ of $\Iw_p$.

\subsection{Locally analytic representations of $\Iw_p$}
We will be brief here, and refer to \cite{Han17} for the details. We write $w_0$ for the longest element in the Weyl group of $\GL_n$, so the weight $w_0\lambda$ is given by $\lambda_n,\lambda_{n-1},\ldots,\lambda_1$. The idea is to define:

\begin{align*}\cA^v_{\lambda} = \Ind_{B(\Zp)}^{\Iw_p}(w_0\lambda)^{v-an} &:= \{f: \Iw_p \to \Qp \mid f(tng) = w_0\lambda(t)f(g) \\&\text{for all }t \in T(\Zp), n \in N(\Zp), \text{and }f|_{\overline{N}_1} \in \cC^{v}_{\mathrm{an}}(\overline{N}_1,\Qp)\},\end{align*} where $\overline{N}_1$ denotes the lower unipotent matrices in $\GL_n(\Zp)$ which map to the identity in $\GL_n(\Fp)$. As a $p$-adic manifold, considering the matrix entries of $\overline{N}_1$ gives an isomorphism to $\Zp^{n(n-1)/2}$, so we can define locally analytic functions on it using polydiscs centred at these points.

We have a natural injective map $V_{\lambda,\Qp} \hookrightarrow \cA_{\lambda}^v$, and we can use $\cA_{\lambda}^v$ to define a local system of Banach spaces on $Y_K$, which again has a natural injective map $\cV_{\lambda,\Qp}\hookrightarrow \cA_{\lambda}^v$.

\begin{exercise}\mbox{}\smallskip\begin{enumerate}
		\item Set $u_p = \mathrm{diag}(p^{n-1},p^{n-2},\ldots,p,1)$ and compute $u_p^{-1}\overline{N}_1u_p$.
		
		\item Show that conjugation by $u_p$ on $\overline{N}_1$ induces a compact map $\cC^{v}_{\mathrm{an}}(\overline{N}_1,\Qp) \to \cC^{v}_{\mathrm{an}}(\overline{N}_1,\Qp)$. 
	\end{enumerate}
\end{exercise}

With some care (\cite[\S2.1]{Han17}) one can show that the spaces $H^i(Y_K,\cA_{\lambda}^v)$ have slope decompositions with respect to the $U_p$ operator (which we rescale as in \ref{ssec:ocquat}, so it is induced by conjugation on $\overline{N}_1$). The idea here is to show that the action of $U_p$ on $H^*(Y_K,\cA_{\lambda}^v)$ lifts to an action of a compact operator $\widetilde{U}_p$ on a complex $\cC^\bullet(Y_K,\cA_{\lambda}^v)$ of potentially ONable $\Qp$-Banach spaces with cohomology $H^*(Y_K,\cA_{\lambda}^v)$. We can then take the slope decomposition on the level of complexes, before passing to cohomology.

There is a classicality theorem:

\begin{thm}(\cite[Theorem 3.2.5]{Han17})
	Suppose $h < \inf_{1\le i \le n-1}(\lambda_i-\lambda_{i+1}+1) $. Then the map $H^*(Y_K,\cV_{\lambda,\Qp})^{\le h} \to H^*(Y_K,\cA_{\lambda}^v)^{\le h}$ is an isomorphism.
\end{thm}
The proof involves a generalisation of the exact sequence in Theorem \ref{thm:bggquat} (a `Bernstein--Gelfand--Gelfand resolution'), which starts:
\[0 \to V_{\lambda,\Qp} \to \cA_{\lambda}^v \to \bigoplus_{w\in W:l(w)=1}\cA^v_{w\cdot\lambda}\]
where the sum is over length one elements in the Weyl group of $\GL_n$ (i.e.~the $n-1$ simple reflections $(i~~i+1)$).

\subsection{Overconvergent cohomology in families}
The natural weight space here parameterizes characters of the (integral) $p$-adic torus:
\begin{defn}
	$\cW_n:= \Hom(T(\Zp),\mathbb{G}_{m,\Qp})$. For an affinoid $(R,R^+)$, we have $\cW_n(R,R^+) = \Hom_{\mathrm{cts}}(T(\Zp),R^\times)$. 
	
	The \emph{integral weights} are the $\Qp$ points of $\cW_n$ given by algebraic characters $\lambda: T \to \Gm$.
\end{defn}
We can describe $\cW_n$ explicitly as a disjoint union of $n$-dimensional open polydiscs.

For an affinoid weight $\kappa: T(\Zp) \to R^\times$, and $v$ sufficiently large, it is possible to define a (bounded) complex of potentially ONable $R$-modules $\cC^\bullet(Y_K,\cA_{\kappa}^v)$ of $\Qp$-Banach spaces equipped with a compact action of $\widetilde{U}_p$, lifting the action of $U_p$ on $H^*(Y_K,\cA_{\kappa}^v)$.

\subsection{The eigenvariety}
We can now construct an eigenvariety. First we construct a Fredholm surface, which is highly non-canonical - it depends on choices involved in the construction of the complex $\cC^\bullet(Y_K,\cA_{\kappa}^v)$. These choices can at least be made independent of $\kappa$. So, we have a compact operator $\widetilde{U}_p$ on 
\[\bigoplus_{i} \cC^i(Y_K,\cA_{\kappa}^v)\]
and its characteristic power series which glues over $\cW$ to give a universal characteristic power series $F_{\widetilde{U}_p}^\dagger \in \cO(\cW_n)\{\{X\}\}$ and a Fredholm hypersurface $\cZ:=V(F_{\widetilde{U}_p}^\dagger) \subset \AA^1_{\cW_n}$.

The slope decompositions of the complex $\cC^\bullet(Y_K,\cA_{\kappa}^v)$ define a perfect complex of coherent sheaves $\cN^{\bullet}$ on $\cZ$. Hansen considers the eigenvariety $\cE(H^*(\cN^\bullet))$ built out of the eigenvariety datum $(\cZ,\oplus_i H^i(\cN^{\bullet}),\TT^S)$. We can alternatively construct a (mildly) `derived' eigenvariety $\cE(\cN^\bullet)$ by considering the finite $\cO_{\cZ}$-algebra generated by $\TT^S$ in $\End_{D^b_{\mathrm{Coh}}(\cZ)}(\cN^\bullet)$.

The natural map \[\End_{D^b_{\mathrm{Coh}}(\cZ)}(\cN^\bullet) \to \End_{\mathrm{Coh}(\cZ)}(\oplus_i H^i(\cN))\] has nilpotent kernel, so it induces a closed immersion $\cE(H^*(\cN^\bullet)) \hookrightarrow \cE(\cN^\bullet)$ which is an isomorphism on reduced subspaces.

\subsection{The example $n = 2$}
When $n = 2$, we have essentially re-constructed the Coleman--Mazur eigencurve, in a way which is very close to the construction using overconvergent modular symbols described in \cite{eigenbook}. The Eichler--Shimura theorem and the arguments sketched in \S\ref{comparison} show that if we restrict everything to the locus $\cW \subset \cW_2$ given by weights which are trivial on the second entry in the torus, then the underlying reduced subspace of $\cE(H^*(\cN^\bullet))$ is equal to the Coleman--Mazur eigencurve. 

One subtlety which arises here is that for open affinoids $U \subset \cW$, the pieces of the slope decomposition $H^1(Y_K,\cA^v_{\kappa_U})^{\le h}$ are not necessarily locally free over $\cO(U)$. They have torsion supported at the points $\kappa \in U$ where $H^0(Y_K,\cA^v_{\kappa})^{\le h} \ne 0$. For example, when $\kappa$ is trivial and $h \ge 1$.  

\subsection{Dimensions of eigenvarieties and classical points}
The eigenvarieties for $\GL_n$ with $n > 2$ are rather mysterious. The weight space $\cW_n$ has dimension $n$, but the classical cuspidal points, by Theorem \ref{clozelcoh}, all have weights lying in the \textit{essentially self-dual subspace} $\cW_n^{esd}$ of dimension $1+\lfloor\frac{n}{2}\rfloor$, or codimension $l_0 = \lfloor\frac{n-1}{2}\rfloor$. 

We expect most irreducible components of the eigenvariety for $\GL_n/\Q$ to have dimension $1+\lfloor\frac{n}{2}\rfloor$. This reflects the fact that `generic' systems of mod $p$ Hecke eigenvalues should only appear in the range of cohomological degrees with length $l_0$ described in Theorem \ref{clozelcoh}. See Urban's paper \cite{urban-eigen} for a precise conjecture (related to a conjecture of Hida in the ordinary case) and also Hansen's paper \cite{Han17} for a conjecture and some evidence. The fact that the dimension is $\ge 1+\lfloor\frac{n}{2}\rfloor$ can be verified in many cases, but equality seems to be much harder and could be regarded as a non-Abelian version of the Leopoldt conjecture. 

The classical points of our eigenvarieties now come from Hecke eigenvalues appearing in $H^*(Y_K,\cV_{\lambda,\Qp})$ for integral weights $\lambda$. They are not expected to be Zariski dense in the eigenvariety when $n > 2$. Examples of Ash--Pollack--Stevens \cite{APS-rigid} in the slope $0$ case demonstrate this. 

\begin{conj}[Ash--Pollack--Stevens]
	Suppose an irreducible component of $\cE(\N^\bullet)$ contains a Zariski dense subset of cuspidal classical points. Then all those classical points come from essentially self-dual automorphic representations $\pi$\footnote{I.e.~$\pi^\vee$ is isomorphic to a twist of $\pi$ by a Hecke character.}.
\end{conj}

Conversely, the essentially self-dual classical points do live (and are Zariski-dense) in closed subsets of dimension $1+\lfloor\frac{n}{2}\rfloor$ \cite{xiang-selfdual}. One can think of these subsets as arising from eigenvarieties for symplectic or orthogonal similitude groups. 

The remaining irreducible components, which are expected to contain the non-essentially-self-dual classical points are `genuinely $p$-adic' in the sense that we can't expect to construct them using classical automorphic representations. This is analogous to the picture when we look at cohomology of locally symmetric spaces with $\Zp$-coefficients --- again the `defect' $l_0$ plays a crucial role, and there are torsion classes which we do not expect to be able to relate to classical automorphic representations.

\renewcommand{\MR}[1]{}
\providecommand{\bysame}{\leavevmode\hbox to3em{\hrulefill}\thinspace}
\providecommand{\MR}{\relax\ifhmode\unskip\space\fi MR }
\providecommand{\MRhref}[2]{%
  \href{http://www.ams.org/mathscinet-getitem?mr=#1}{#2}
}
\providecommand{\href}[2]{#2}

\end{document}